\documentclass{amsart}
\usepackage{graphicx}
\usepackage{amssymb}
\usepackage{amsmath}
\usepackage{amsthm}
\usepackage{subfigure}

\newtheorem{thm}{Theorem}[section]

\newtheorem{lem}[thm]{Lemma}
\newtheorem{prop}[thm]{Proposition}
\theoremstyle{definition}
\newtheorem{defn}[thm]{Definition}
\newtheorem{rem}[thm]{Remark}

\newtheorem*{defn*}{Definition}
\newtheorem*{rems*}{Remarks}
\newtheorem*{rem*}{Remark}

\numberwithin{equation}{section}

\DeclareMathOperator {\Symp} {Symp}

\DeclareMathOperator {\Span} {span}

\def \algrest {\left [\Symp (\mathbb R^{2n})\right ]_{N}}

\def \algrestall {\bigl [\Lambda ^2(\mathbb R^{2n})\bigr ]_N}

\def \algrestclosed {\bigl [ Z ^2(\mathbb R^{2n})\bigr ]_N}

\begin{document}

\title[Symplectic $W_8, W_9$ singularities] {Symplectic $W_8$ and $W_9$ singularities}

\author{\.{Z}aneta Tr\c{e}bska}
\address{Warsaw University of Technology\\
Faculty of Mathematics and Information Science\\
Plac Politechniki 1\\
00-661 Warsaw\\
Poland\\}

\email{ztrebska@mini.pw.edu.pl}




\maketitle

\begin{abstract}
We study the local symplectic algebra of curves. We use the method
of algebraic restrictions to classify symplectic $W_8$ and $W_9$
singularities. We use discrete symplectic invariants  to distinguish symplectic singularities of the curves. We also give the geometric description of symplectic classes.
\end{abstract}

\section{Introduction}

In this paper we study the symplectic classification of singular curves under the following equivalence:

\begin{defn} \label{symplecto}
Let $N_1, N_2$ be germs of subsets of symplectic space $(\mathbb{R}^{2n}, \omega)$. $N_1, N_2$ are \textbf{symplectically equivalent} if there exists a symplectomorphism-germ $\Phi:(\mathbb{R}^{2n}, \omega) \rightarrow(\mathbb{R}^{2n}, \omega)$ such that $\Phi(N_1)=N_2$.
\end{defn}

We recall that  $\omega$ is a symplectic form if $\omega$ is a
smooth nondegenerate closed 2-form, and $\Phi:\mathbb{R}^{2n}
\rightarrow\mathbb{R}^{2n}$ is a symplectomorphism if $\Phi$ is
diffeomorphism and $\Phi ^* \omega=\omega$.

\medskip

Symplectic classification of curves was first studied by V. I.
Arnold. In \cite{Ar1} and \cite{Ar2} the author studied singular curves in symplectic and contact spaces and introduced the local symplectic and contact algebra. In \cite{Ar2} V. I. Arnold discovered new symplectic
invariants of singular curves. He proved that the $A_{2k}$
singularity of a planar curve (the orbit with respect to standard
$\mathcal A$-equivalence of parameterized curves) split into
exactly $2k+1$ symplectic singularities (orbits with respect to
symplectic equivalence of parameterized curves). He distinguished different symplectic singularities by different orders of tangency of the parameterized curve to the \emph{nearest} smooth Lagrangian submanifold. Arnold posed a problem of expressing these invariants in terms of the local algebra's interaction with the symplectic structure and he proposed to call this interaction the local symplectic algebra.

\medskip

In \cite{IJ1} G. Ishikawa and S. Janeczko classified symplectic singularities of curves in the $2$-dimensional symplectic space. All simple curves in this classification are quasi-homogeneous.

\medskip

We recall that a subset $N$ of $\mathbb R^m$ is \textbf{quasi-homogeneous} if there exists a coordinate system $(x_1,\cdots,x_m)$ on $\mathbb R^m$ and
positive numbers $w_1,\cdots,w_m$ (called weights) such that for any
point $(y_1,\cdots,y_m)\in \mathbb R^m$ and any $t>0$
if $(y_1,\cdots,y_m)$ belongs to $N$ then the point
$(t^{w_1}y_1,\cdots,t^{w_m}y_m)$ belongs to $N$.

\medskip

The generalization of results in \cite{IJ1} to volume-preserving classification of singular varieties and maps  in arbitrary dimensions was obtained in \cite{DR}. A symplectic form on a $2$-dimensional manifold is a special case of a volume form on a smooth manifold. 

In \cite{K} P. A. Kolgushkin classified the stably simple symplectic singularities of parameterized curves (in the $\mathbb C$-analytic category). Symplectic singularity is stably simple if it is simple and remains simple if the ambient symplectic space is symplectically embedded (i.e. as a symplectic submanifold) into a larger symplectic space. 


In \cite{Zh} was developed the local contact algebra. The main results were based on the notion of the algebraic restriction of a contact structure to a subset $N$ of a contact manifold.

In \cite{DJZ2} new symplectic invariants of singular quasi-homogeneous  subsets of a symplectic space were explained by the algebraic restrictions of the symplectic form to these subsets.

\smallskip

The algebraic restriction is an equivalence class of the following relation on the space of differential $k$-forms:

Differential $k$-forms $\omega_1$ and $\omega_2$ have the same
{\bf algebraic restriction} to a subset $N$ if
$\omega_1-\omega_2=\alpha+d\beta$, where $\alpha$ is a $k$-form
vanishing on $N$ and $\beta$ is a $(k-1)$-form vanishing on $N$.

\smallskip

In \cite{DJZ2} the generalization of Darboux-Givental theorem (\cite{ArGi})
to germs of arbitrary subsets of the symplectic space was obtained. This result reduces
the problem of symplectic classification of germs of quasi-homo\-ge\-neous subsets to
the problem of classification of algebraic restrictions of symplectic
forms to these subsets. For non-quasi-homogeneous subsets there is one more cohomological invariant except the algebraic restriction (\cite{DJZ2}, \cite{DJZ1}). The dimension of the space of algebraic restrictions of closed $2$-forms to a $1$-dimensional quasi-homogeneous isolated complete
intersection singularity $C$ is equal to the multiplicity of $C$ (\cite{DJZ2}). In \cite{D} it was proved that the space of algebraic restrictions of closed $2$-forms to a $1$-dimensional (singular) analytic variety is finite-dimensional.
In \cite{DJZ2} the method of algebraic restrictions was applied to various classification problems in a symplectic space. In particular the complete symplectic classification of classical $A-D-E$ singularities of planar curves and $S_5$ singularity were obtained. Most of different symplectic singularity classes were distinguished by new discrete symplectic invariants: the index of isotropy and the symplectic multiplicity.

\smallskip



In this paper using the method of algebraic restrictions we obtain the complete symplectic classification of the classical isolated complete intersection singularities $W_8$ and $W_9$  (Theorems \ref{w8-main} and \ref{w9-main}). We
calculate discrete symplectic invariants for this classification (Theorems \ref{lagr-w8} and \ref{lagr-w9}) and we present geometric descriptions of its symplectic orbits (Theorems  \ref{geom-cond-w8} and \ref{geom-cond-w9}).

\smallskip

Isolated complete intersection singularities were intensively studied by many authors (e. g. see \cite{L}), because of their interesting geometric, topological and algebraic properties.

\smallskip

In \cite{DT1} following ideas from \cite{Ar1} and \cite{D}  new
discrete symplectic invariants - the Lagrangian tangency orders
were introduced and used to distinguish symplectic
singularities of simple planar curves of type $A-D-E$, symplectic
 $T_7$ and $T_8$ singularities. 
 In \cite{DT2} was obtained the complete symplectic classification of
the isolated complete intersection singularities $S_{\mu}$
for $\mu>5$. 

\medskip

The paper is organized as follows.  In Section \ref{discrete} we recall  discrete symplectic invariants (the symplectic multiplicity, the index of isotropy and the Lagrangian tangency orders).
Symplectic classification of $W_8$ and $W_9$ singularity is studied in Sections \ref{sec-w8} and \ref{sec-w9} respectively.
In Section \ref{proofs} we recall the method of algebraic restrictions and use it to classify  symplectic singularities.

\section{Discrete symplectic invariants.}\label{discrete}

We can use discrete symplectic invariants to characterize
symplectic singularity classes.

 The first invariant is a symplectic
multiplicity (\cite{DJZ2}) introduced  in \cite{IJ1} as a
symplectic defect of a curve.

\medskip

Let $N$ be a germ of a subvariety of $(\mathbb R^{2n},\omega)$.

\begin{defn}
\label{def-mu}
 The {\bf symplectic multiplicity} $\mu_{symp}(N)$ of  $N$ is the codimension of
 the symplectic orbit of $N$ in the orbit of $N$ with respect to the action of the group of local diffeomorphisms.
\end{defn}

The second invariant is the index of isotropy \cite{DJZ2}.

\begin{defn}
The {\bf index of isotropy} $ind(N)$ of $N$ is the maximal
order of vanishing of the $2$-forms $\omega \vert _{TM}$ over all
smooth submanifolds $M$ containing $N$.
\end{defn}

This invariant has geometrical interpretation. An equivalent definition is as follows: the index of isotropy of $N$ is the maximal order of tangency between non-singular submanifolds containing $N$ and non-singular isotropic submanifolds of the same dimension.  The index of isotropy is equal to $0$ if $N$ is not contained in any non-singular submanifold which is tangent to some isotropic submanifold of the same dimension. If $N$ is contained in a non-singular Lagrangian submanifold then the index of isotropy is $\infty $.


The symplectic multiplicity and the index of isotropy can be described in terms of algebraic restrictions (Propositions \ref{sm} and \ref{ii} in Section \ref{proofs}).




\medskip

There is one more discrete symplectic invariant introduced in \cite{D} following ideas from \cite{Ar2} which is defined specifically for a parameterized curve. This is the maximal
tangency order of a curve $f:\mathbb R\rightarrow M$ to a smooth Lagrangian submanifold. If $H_1=...=H_n=0$ define a smooth submanifold $L$ in the symplectic space then the tangency order of
a curve $f:\mathbb R\rightarrow M$ to $L$ is the minimum of the orders of vanishing at $0$ of functions $H_1\circ f,\cdots, H_n\circ f$. We denote the tangency order of $f$ to $L$ by $t(f,L)$.

\begin{defn}
The {\bf Lagrangian tangency order} $Lt(f)$\textbf{ of a curve} $f$ is the
maximum of $t(f,L)$ over all smooth Lagrangian submanifolds $L$ of
the symplectic space.
\end{defn}

The Lagrangian tangency order of the quasi-homogeneous curve in a symplectic space can also be  expressed in terms of algebraic restrictions  (Proposition \ref{lto} in Section \ref{proofs}).


\medskip

In \cite{DT1}  the above invariant was generalized for germs of
curves and multi-germs of curves which may be parameterized
analytically since the Lagrangian tangency order is the same for every
'good' analytic parameterization of a curve.

\medskip

Consider a multi-germ $(f_i)_{i\in\{1,\cdots,r\}}$ of analytically
parameterized curves $f_i$.  We have $r$-tuples $(t(f_1,L), \cdots,
t(f_r,L))$ for any smooth submanifold $L$ in the
symplectic space.

\begin{defn}
For any $I\subseteq \{1,\cdots, r\}$ we define \textbf{the tangency order of the multi-germ } $(f_i)_{i\in I}$ to $L$:
$$t[(f_i)_{i\in\ I},L]=\min_{i\in\ I} t(f_i,L).$$
\end{defn}

\begin{defn}
The {\bf Lagrangian tangency order} $Lt((f_i)_{i\in\ I})$ \textbf{of a multi-germ } $(f_i)_{i\in I}$ is the maximum of $t[(f_i)_{i\in\ I},L]$ over all smooth Lagrangian submanifolds $L$ of the symplectic space.
\end{defn}



\section{Symplectic $W_8$-singularities}\label{sec-w8}

 Denote by $(W_8)$  the class of varieties in  a fixed symplectic space $(\mathbb R^{2n}, \omega )$ which are diffeomorphic to
\begin{equation}
\label{defw8} W_8=\{x\in \mathbb R ^{2n\geq
4}\,:x_1^2+ x_3^3=x_2^2+x_1x_3=x_{\geq 4}=0\}.\end{equation}

\noindent This is simple  $1$-dimensional isolated complete intersection singularity $W_8$ (\cite{G}, \cite{AVG}\footnote{ There is a mistake  in description of $W_8$ singularity in \cite{AVG}.  We find there \par $W_8=\{x\in\mathbb R^{2n\geq
4}\,:x_1^2+ x_2^3=x_2^2+x_1x_3=x_{\geq 4}=0\}$ which is not an isolated complete intersection singularity.}).
Here $N$ is a quasi-homogeneous set with weights $w(x_1)=6,\; w(x_2)=5$,\; $w(x_3)=4$.

\medskip

We used the method of algebraic restrictions to obtain a complete classification of symplectic singularities of $(W_8)$ presented in the following theorem.

\begin{thm}\label{w8-main}
Any submanifold of the symplectic space $(\mathbb R^{2n},\sum_{i=1}^n dp_i \wedge dq_i)$ where $n\geq3$ (resp. $n=2$) which is diffeomorphic to $W_8$ is symplectically equivalent to one and only one of the normal forms $W_8^i, i = 0,1,\cdots ,8$ (resp. $i=0,1,2a,2b)$  listed below. The parameters $c, c_1, c_2$ of the normal forms are moduli.

\medskip
\begin{small}

\noindent $W_8^0: \ p_1^2 + p_2q_1 = 0, \ \ p_2^2 +q_1^3= 0, \
\ q_2 = c_1q_1 + c_2p_1, \ \ p_{\ge 3} = q_{\ge 3} = 0$;

\medskip

\noindent $W_8^1: \ q_1^2 + p_1q_2 = 0, \ \ p_1^2 + q_2^3= 0, \
\ p_2 = c_1p_1 + c_2q_1q_2, \ \ p_{\ge 3} = q_{\ge 3} = 0, \; c_1\ne 0$;

\medskip

\noindent $W_8^{2a}: \ p_2^2 \pm p_1q_1 = 0, \ \ p_1^2+q_1^3 = 0, \
\ q_2 = \frac{c_1}{2}q_1^2+\frac{c_2}{3}q_1^3, \ \ p_{\ge 3} = q_{\ge 3} = 0$;

\medskip

\noindent $W_8^{2b}: \ q_1^2 + p_1q_2 = 0, \ \ p_1^2 + q_2^3= 0, \
\ p_2 = c_1q_1q_2 + \frac{c_2}{2}q_1^2, \ \ p_{\ge 3} = q_{\ge 3} = 0$;

\medskip

\noindent $W_8^3: \ p_2^2+p_1p_3=0, \  p_1^2+p_3^3=0, \ q_1=q_2=0, q_3=-p_2p_3-\frac{c_1}{2}p_2^2-c_2p_1p_2, \  p_{>3}=q_{>3}=0$;

\medskip

\noindent $W_8^4: \ p_2^2 + p_1p_3 = 0,  \ p_1^2 +p_3^3 = 0, \ q_1 =  q_2=0, q_3= \mp \frac{1}{2}p_2^2-c_1p_1p_2-c_2p_2p_3^2,  \ p_{>3} = q_{>3} = 0$;

\medskip

\noindent $W_8^5: \ p_2^2 + p_1p_3 = 0,  \ p_1^2 +p_3^3 = 0, \ q_1 =  q_2=0, q_3= -p_1p_2-cp_2p_3^2,  \ p_{>3} = q_{>3} = 0$;

\medskip

\noindent $W_8^6: \ p_2^2 + p_1p_3 = 0,  \ p_1^2 +p_3^3 = 0, \ q_1 =  q_2=0, q_3= -p_2p_3^2-\frac{c}{3}p_2^3,  \ p_{>3} = q_{>3} = 0$;

\medskip

\noindent $W_8^7: \ p_2^2 + p_1p_3 = 0,  \ p_1^2 +p_3^3 = 0, \ q_1 =  q_2=0, q_3= -\frac{1}{3}p_2^3,  \ p_{>3} = q_{>3} = 0$;

\medskip

\noindent $W_8^8: \ p_2^2 + p_1p_3 = 0,  \ p_1^2 +p_3^3 = 0,  \ p_{>3} = q_{>0} = 0$.
\end{small}
\end{thm}

\medskip

 In Section \ref{w8-lagr} we use symplectic invariants (in particular the Lagrangian tangency order) to distinguish  symplectic
singularity classes. In Section \ref{w8-geom_cond} we propose a geometric description of these singularities which confirms the classification. Some of the proofs are presented in Section \ref{proofs}.

\subsection{Distinguishing symplectic classes of $W_8$ by Lagrangian tangency order and the index of isotropy}
\label{w8-lagr}

A curve $N\in (W_8)$ can be described as a  parameterized curve $C(t)$. Its
parameterization is given in the second column of Table \ref{tabw8-lagr}. To characterize the symplectic classes  we use the following invariants:
\begin{itemize}
  \item $L_N=Lt(N)=\max\limits _L (t(C(t),L)),$ \par where $L$ is a smooth Lagrangian submanifold of the symplectic space,
   \item $ind$ - the index of isotropy of $N$.
  \end{itemize}

\begin{rem}
 The invariants can be calculated knowing algebraic restrictions for symplectic classes. We used Proposition \ref{ii} to calculate the index of isotropy. The Lagrangian tangency order we can calculate using Proposition \ref{lto} or by  applying directly the definition of the Lagrangian tangency order and finding a Lagrangian submanifold nearest to the curve $C(t)$.
\end{rem}

\begin{thm}
\label{lagr-w8} A stratified submanifold $N\in (W_8)$ of a symplectic space $(\mathbb R^{2n}, \omega_0)$ with the canonical coordinates $(p_1, q_1, \cdots, p_n, q_n)$ is symplectically equivalent to one and only one of the curves presented in the second column of Table \ref{tabw8-lagr}. The parameters $c, c_1, c_2$ are moduli. The  index of isotropy and the Lagrangian tangency order of the curve $N$ are presented respectively in the third and fourth column of Table \ref{tabw8-lagr}.
\end{thm}

\renewcommand*{\arraystretch}{1.3}
\begin{center}
\begin{table}[h]

    \begin{small}
    \noindent
    \begin{tabular}{|p{2.5cm}|p{6.cm} |c|c|}
                      \hline
    Class &  Parameterization of $N$     & $ind$ & $L_N$     \\ \hline

   $(W_8)^0$ \; \; $2n\ge 4$ & $(t^5,-t^4,t^6,-c_1t^4+c_2t^5,0,\cdots )$  & $0$ & $5$  \\ \hline

     $(W_8)^1$ \; \; $2n\ge 4$ & $(t^6,t^5,c_1 t^6-c_2t^9,-t^4,0,\cdots )$  & $0$ & $6$ \\ \hline

    $(W_8)^{2a}$ \; $2n\ge 4$ & $(\pm t^6,-t^4,t^5,\frac{c_1}{2}t^8-\frac{c_2}{3}t^{12},0,\cdots )$  & $0$ & $6$ \\ \hline

    $(W_8)^{2b}$ \; $2n\ge 4$ & $(t^6,t^5,-c_1t^9+\frac{c_2}{2}t^{10},-t^4,,0,\cdots )$ & $0$ & $6$ \\ \hline

    $(W_8)^3$ \; \; $2n\ge 6$ & $(t^6,0,t^5,0, -t^4,t^9-\frac{c_1}{2}t^{10}-c_2t^{11},0,\cdots )$  & $1$ & $9$ \\ \hline

     $(W_8)^4$ \; \; $2n\ge 6$ & $(t^6,0,t^5,0, -t^4,\mp t^{10}-c_1t^{11}-c_2t^{13},0,\cdots )$  & $1$ & $10$ \\ \hline

    $(W_8)^5$ \; \; $2n\ge 6$ & $(t^6,0,t^5,0, -t^4,-t^{11}-ct^{13},0,\cdots )$ & $1$ & $11$ \\ \hline

     $(W_8)^6$ \; \; $2n\ge 6$ & $(t^6,0,t^5,0, -t^4,-t^{13}-\frac{c}{3}t^{15},0,\cdots )$  & $2$ & $13$ \\ \hline

     $(W_8)^7$ \; \; $2n\ge 6$ & $(t^6,0,t^5,0, -t^4,-t^{15},0,\cdots )$  & $2$ & $15$ \\ \hline

   $(W_8)^8$ \; \; $2n\ge 6$ & $(t^6,0,t^5,0, -t^4,0,0,\cdots )$   & $\infty$  & $\infty$  \\ \hline

\end{tabular}

\smallskip

\caption{\small The symplectic invariants for symplectic classes of $W_8$ singularity.}\label{tabw8-lagr}

\end{small}
\end{table}
\end{center}

\medskip

\begin{rem}The comparison of invariants presented in Table \ref{tabw8-lagr} shows  that the Lagrangian tangency order distinguishes more symplectic classes than the index of isotropy. Symplectic classes $(W_8)^{2a}$ and $(W_8)^{2b}$ can not be distinguished by any of the invariants but we can distinguish them by geometric conditions.
\end{rem}

\medskip


\subsection{Geometric conditions for the classes $(W_8)^i$}
\label{w8-geom_cond}


We can characterize the symplectic classes $(W_8)^i$ by geometric conditions without using any local coordinate system.
\medskip

Let $N\in (W_8)$. Denote by $W$ the tangent space at $0$ to some (and then any) non-singular $3$-manifold containing $N$. We can define the following subspaces of this space:
  $\ell$ -- the tangent line at $0$ to the curve $N$,
  $V$ -- the $2$-space tangent at $0$ to the curve $N$.

\smallskip

The classes $(W_8)^i$ satisfy special conditions in terms of the restriction $\omega\vert_ W $, where $\omega$ is the symplectic form.
For $N=W_8=$(\ref{defw8}) it is easy to calculate
\begin{equation} \label{linesw8}
W\!=\!\Span (\partial /\partial x_1,\partial /\partial x_2,\partial /\partial x_3),  \
\ell\!=\!\Span (\partial /\partial x_3), \
V\!=\!\Span (\partial /\partial x_2,\partial /\partial x_3).\end{equation}


\begin{thm}
\label{geom-cond-w8}If a stratified submanifold $N\in (W_8)$ of a
symplectic space $(\mathbb R^{2n}, \omega)$ belongs to the class
$(W_8)^i$ then the couple $(N, \omega)$ satisfies the corresponding
conditions in the last column of Table \ref{tabw8-geom}.

\end{thm}


\renewcommand*{\arraystretch}{1.3}
\begin{center}
\begin{table}[h]
    \begin{small}
    \noindent
    \begin{tabular}{|p{1.2cm}|p{5.5cm}|p{4cm}|}
            \hline
    Class &  Normal form & Geometric conditions  \\ \hline

  $(W_8)^0$   &  $[W_8]^0: [\theta _1 + c_1\theta _2 + c_2\theta _3]_{W_8}$
    &  $  \omega|_V \neq 0$ 
    \\  \hline

 $(W_8)^1$ & $[W_8]^1: [c_1\theta _2 + \theta _3 + c_2\theta _4]_{W_8}$,\;$c_1\ne 0$ & $\omega|_V=0$ and $\ker\omega \ne \ell$    \\ \hline

 $(W_8)^{2a}$ & $[W_8]^{2a}: [\pm\theta _2 + c_1\theta _4 + c_2\theta _7]_{W_8}$ & $\omega|_V=0$ and $\ker\omega \ne \ell$ \\ \hline

 $(W_8)^{2b}$ & $[W_8]^{2b}: [\theta _3 + c_1\theta _4 + c_2\theta _5]_{W_8}$ & $\omega|_V=0$ and $\ker\omega = \ell$ \\ \hline 

 & & $\omega\vert_ W = 0$\\ \hline
$(W_8)^3$  & $[W_8]^3: [\theta _4 + c_1\theta_5+c_2\theta_6]_{W_8}$
                &  $L_N=9$  \\ \hline
$(W_8)^4$ & $[W_8]^4: [\pm\theta _5 + c_1\theta_6+c_2\theta_7]_{W_8}$
                &  $L_N=10$   \\ \hline
$(W_8)^5$ & $[W_8]^5: [\theta_6+c\theta_7]_{W_8}$
                &  $L_N=11$   \\ \hline
$(W_8)^6$ & $[W_8]^6: [\theta_7+c\theta_8]_{W_8}$
                &  $L_N=13$   \\ \hline
$(W_8)^7$ & $[W_8]^7: [\theta_8]_{W_8}$
                &  $L_N=15$   \\ \hline
$(W_8)^8$ & $[W_8]^8: [0]_{W_8}$
                &   $N$ is contained in a smooth Lagrangian submanifold   \\ \hline
    \end{tabular}

\smallskip

\caption{\small Geometric interpretation of singularity classes of $W_8$: $W$ is the tangent space to a non-singular 3-dimensional manifold in $(\mathbb R^{2n\geq4}, \omega)$ containing $N\in(W_8)$. }\label{tabw8-geom}

\end{small}
\end{table}
\end{center}

\begin{proof}[Sketch of the proof of Theorem \ref{geom-cond-w8}]
 We have to show that the conditions in the row of $(W_8)^i$ are satisfied for any $N\in (W_8)^i$.
\smallskip

\noindent  Each of the conditions in the last column of Table \ref{tabw8-geom}
is invariant with respect to the action of the group of diffeomorphisms in the space of pairs $(N,\omega)$.
 Because each of these conditions depends only on the algebraic restriction $[\omega ]_N$ we can take the simplest $2$-forms $\omega ^i$ representing
the normal forms $[W_8]^i$ for algebraic restrictions:
 $\omega ^0, \ \omega ^1, \ \omega ^{2,a},\ \omega ^{2,b}, \ \omega ^3, \ \omega ^4, \ \omega ^5, \ \omega ^6, \ \omega ^7, \ \omega ^8$ and we can check that the pair $(W_8,\omega=\omega ^i)$ satisfies the condition
in the last column of Table \ref{tabw8-geom}.

We note that in the case $N = W_8 = (\ref{defw8})$ one has the description (\ref{linesw8}) 
of the subspaces $W, \ell$ and $V$. By simple calculation and observation of the Lagrangian tangency orders we obtain that the conditions corresponding to the classes $(W_8)^i$ are satisfied.
\end{proof}

\smallskip


\section{Symplectic $W_9$-singularities}\label{sec-w9}

 Denote by $(W_9)$  the class of varieties in  a fixed symplectic space $(\mathbb R^{2n}, \omega )$ which are diffeomorphic to
\begin{equation}
\label{defw9} W_9=\{x\in \mathbb R ^{2n\geq
4}\,:x_1^2+ x_2x_3^2=x_2^2+x_1x_3=x_{\geq 4}=0\}.\end{equation}

\noindent This is simple  $1$-dimensional isolated complete intersection singularity $W_9$ (\cite{G}, \cite{AVG}).
Here $N$ is quasi-homogeneous with weights $w(x_1)\!=\!5,\, w(x_2)\!=\!4,\, w(x_3)\!=\!3$.

\medskip

We present a complete classification of symplectic singularities of $(W_9)$ which was obtained using the method of algebraic restrictions.

\begin{thm}\label{w9-main}
Any submanifold of the symplectic space $(\mathbb R^{2n},\sum_{i=1}^n dp_i \wedge dq_i)$ where $n\geq3$ (resp. $n=2$) which is diffeomorphic to $W_9$ is symplectically equivalent to one and only one of the normal forms $W_9^i, i = 0,1,\cdots ,9$ (resp. $i=0,1,2$) listed below. The parameters $c, c_1, c_2$ of the normal forms are moduli.

\medskip
\begin{small}

\noindent $W_9^0: \ p_1^2 + p_2q_2^2 = 0, \ \ p_2^2 +p_1q_2= 0, \
\ q_1 = c_1q_2 + c_2p_2, \ \ p_{\ge 3} = q_{\ge 3} = 0$;

\medskip

\noindent $W_9^1: \ p_1^2 + p_2q_1^2 = 0, \ \ p_2^2 \pm p_1q_1= 0, \
\ q_2 = -c_1p_1 + \frac{c_2}{2}q_1^2, \ \ p_{\ge 3} = q_{\ge 3} = 0$;

\medskip

\noindent $W_9^2: \ p_1^2 + q_1p_2^2 = 0, \ \ q_1^2+p_1p_2 = 0, \
\ q_2 = c_1q_1p_2-c_2p_1p_2, \ \ p_{\ge 3} = q_{\ge 3} = 0$;

\medskip

\noindent $W_9^3: \ p_1^2 + p_2p_3^2 = 0,  \ p_2^2 +p_1p_3 = 0,  q_3=\mp p_2p_3-c_1p_1p_3-c_2p_1p_2,  \ q_1 =  q_2=p_{>3} = q_{>3} = 0$;

\medskip

\noindent $W_9^4: \ p_1^2 + p_2p_3^2 = 0,  \ p_2^2 +p_1p_3 = 0, \  q_3= -p_1p_3-c_1p_1p_2-c_2p_2p_3^2, q_1=q_2=p_{>3}=q_{>3} = 0$;

\medskip

\noindent $W_9^5: \ p_1^2 + p_2p_3^2 = 0,  \ p_2^2 +p_1p_3 = 0, \  q_3=\mp p_1p_2-c_1p_2p_3^2-c_2p_1p_3^2, q_1=q_2=p_{>3}=q_{>3}=0$;

\medskip

\noindent $W_9^6: \ p_1^2 + p_2p_3^2 = 0,  \ p_2^2 +p_1p_3 = 0, \ q_1 =  q_2=0, q_3= -p_2p_3^2-cp_1p_3^2,  \ p_{>3} = q_{>3} = 0$;

\medskip

\noindent $W_9^7: \ p_1^2 + p_2p_3^2 = 0,  \ p_2^2 +p_1p_3 = 0, \ q_1 =  q_2=0, q_3=\mp p_1p_3^2-cp_2p_3^3,  \ p_{>3} = q_{>3} = 0$;

\medskip

\noindent $W_9^8: \ p_1^2 + p_2p_3^2 = 0,  \ p_2^2 +p_1p_3 = 0, \ q_1 =  q_2=0, q_3=\mp p_2p_3^3,  \ p_{>3} = q_{>3} = 0$;

\medskip

\noindent $W_9^9: \ p_1^2 + p_2p_3^2 = 0,  \ p_2^2 +p_1p_3 = 0,  \ p_{>3} = q_{>0} = 0$.
\end{small}
\end{thm}

 In Section \ref{w9-lagr} we use  the Lagrangian tangency orders to distinguish  symplectic classes. In Section \ref{w9-geom_cond} we propose a geometric description of the symplectic singularities. Some of the proofs are presented in Section \ref{proofs}.

\subsection{Distinguishing symplectic classes of $W_9$ by Lagrangian tangency orders}
\label{w9-lagr}
Lagrangian tangency orders were used to distinguish symplectic classes of $(W_9)$. A curve $N\in (W_9)$ may be described as a union of two parameterized branches:  $C_1$ and $C_2$. The curve $C_1$ is nonsingular and the curve $C_2$ is singular. Their parameterization  in the coordinate system $(p_1,q_1,p_2,q_2,\cdots,p_n,q_n)$ is presented in the second column of Tables \ref{tabw9-lagr}. To characterize the symplectic classes of this singularity we use the following two invariants:
\begin{itemize}
  \item $L_N=Lt(C_1,C_2)=\max\limits _L (\min \{t(C_1,L),t(C_2,L)\}),$
  \item $L_{2}=Lt(C_2)=\max\limits _L\, t(C_2,L),$
\end{itemize}
where $L$ is a smooth Lagrangian submanifold of the symplectic space.



\begin{thm}
\label{lagr-w9} A stratified submanifold $N\in (W_9)$ of a symplectic space $(\mathbb R^{2n}, \omega_0)$ with the canonical coordinates $(p_1, q_1, \cdots, p_n, q_n)$ is symplectically equivalent to one and only one of the curves presented in the second column of Table \ref{tabw9-lagr}. The parameters $c, c_1, c_2$ are moduli. The Lagrangian tangency orders  are presented in  the third and fourth  column of Table \ref{tabw9-lagr}.
\end{thm}

\renewcommand*{\arraystretch}{1.4}
\begin{center}
\begin{table}[h]

    \begin{footnotesize}
    \noindent
    \begin{tabular}{|p{0.8cm}|p{9.3cm} |c|c|}
                      \hline
    Class &  Parameterization of branches      & $L_N$    & $L_2$ \\ \hline

 $(W_9)^0$ & $C_1:(0,c_1t,0,t,0,0,\cdots )$, \; \; $C_2:(t^5,-c_1t^3-c_2t^4,-t^4,-t^3,0,\cdots )$     & $4$  & $4$ \\
    \hline

$(W_9)^1$ & $C_1:(0,\pm t,0,\frac{c_2}{2}t^2,0,\cdots ),$ \ \ $C_2:(t^5,\mp t^3,-t^4,-c_1 t^5+\frac{c_2}{2}t^6,0,\cdots )$   & $5$  & $5$ \\
 \hline

$(W_9)^2$ & $C_1:(0,0, t,0,0,\cdots ),$ \;  \ $C_2:(t^5,-t^4,-t^3,-c_1t^7+c_2t^8,0,\cdots )$   & $5$  & $5$ \\
 \hline 

$(W_9)^3$ & $C_1:(0,0,0,0,t,0,\cdots ),$ \ $C_2\!:(t^5,0,-t^4,0, t^3,\mp t^7+c_1t^8+c_2t^9,0,\cdots )$    & $7$  & $7$ \\
 \hline

$(W_9)^4$ & $C_1:(0,0,0,0,t,0,\cdots ),$ \ $C_2\!:(t^5,0,-t^4,0, t^3,t^8+c_1t^9+c_2t^{10},0,\cdots )$    & $8$  & $8$ \\
 \hline

$(W_9)^5$ & $C_1\!:\!(0,0,0,0,t,0,\cdots),$ \  $C_2\!:(t^5\!,0,-t^4\!,0,t^3\!,\pm t^9+c_1t^{10}-c_2t^{11},0,\cdots )$    & $9$  & $\infty$ \\
\hline

$(W_9)^6$ & $C_1:(0,0,0,0,t,0,\cdots ),$ \; $C_2:(t^5,0,-t^4,0, t^3,t^{10}-ct^{11},0,\cdots )$    & $10$  & $\infty$ \\
 \hline

$(W_9)^7$ & $C_1:(0,0,0,0,t,0,\cdots ),$ \; $C_2:(t^5,0,-t^4,0, t^3,\mp t^{11}-ct^{13},0,\cdots )$    & $11$  & $\infty$ \\
 \hline

$(W_9)^8$ & $C_1:(0,0,0,0,t,0,\cdots ),$ \; $C_2:(t^5,0,-t^4,0, t^3,\mp t^{13},0,\cdots )$   & $13$  & $\infty$ \\
  \hline

$(W_9)^9$ & $C_1:(0,0,0,0,t,0,\cdots ),$ \; $C_2:(t^5,0,-t^4,0, t^3,0,0,\cdots )$   & $\infty$  & $\infty$ \\
 \hline

\end{tabular}

\smallskip

\caption{\small  Lagrangian tangency orders for symplectic classes of $W_9$ singularity.}\label{tabw9-lagr}
\end{footnotesize}
\end{table}
\end{center}
\begin{rem}
 The invariants can be calculated knowing the paramererization of branches $C_1$ and $C_2$.  We   apply directly the definition of the Lagrangian tangency order  finding a Lagrangian submanifold nearest to the  branches.
\end{rem}

\subsection{Geometric conditions for the classes $(W_9)^i$}
\label{w9-geom_cond}

\parbox{5cm}{}

\medskip

Let $N\in (W_9)$. Denote by $W$ the tangent space at $0$ to some (and then any) non-singular $3$-manifold containing $N$. We can define the following subspaces of this space:

  $\ell$ -- the tangent line at $0$ to both branches of $N$,

  $V$ -- $2$-space tangent at $0$ to the singular branch of $N$.

\smallskip

The classes $(W_9)^i$ satisfy special conditions in terms of the restriction $\omega\vert_ W $, where $\omega $ is the symplectic form.

\begin{thm}
\label{geom-cond-w9} A stratified submanifold $N\in (W_9)$ of a symplectic space $(\mathbb R^{2n}, \omega )$ belongs to the class $(W_9)^i$ if and only if the couple $(N, \omega )$ satisfies the corresponding conditions in the last column of Table \ref{tabw9-geom}.
\end{thm}

\linespread{1.1}

\renewcommand*{\arraystretch}{1.2}
\begin{center}
\begin{table}[h]
    \begin{small}
    \noindent
    \begin{tabular}{|p{1.0cm}|p{4cm}|p{6.5cm}|}
            \hline
    Class &  Normal form & Geometric conditions  \\ \hline

  $(W_9)^0$   &  $[W_9]^0: [\theta _1 + c_1\theta _2 + c_2\theta _3]_{W_9}$
    &  $  \omega|_V \neq 0$ (2-space tangent to $N$ is not isotropic)  \\  \hline
  $(W_9)^1$ & $[W_9]^1: [\pm\theta _2 + c_1\theta _3 + c_2\theta _4]_{W_9}$ & $\omega|_V=0$ and $\ker\omega \ne \ell$    \\ \hline

 $(W_9)^{2}$ & $[W_9]^2: [\theta _3 + c_1\theta_4+c_2\theta_5]_{W_9}$ & $\omega|_V=0$ and $\ker\omega = \ell$ \\ \hline 

 & & $\omega\vert_ W = 0$\\ \hline
$(W_9)^3$  & $[W_9]^3: [\pm\theta _4 +c_1\theta_5+c_2\theta_6 ]_{W_9}$
                &  $L_N=7$  \\ \hline
$(W_9)^4$ & $[W_9]^4: [\theta _5 + c_1\theta _6+c_2\theta _7]_{W_9}$
                &  $L_N=8$   \\ \hline
$(W_9)^5$ & $[W_9]^5: [\pm\theta _6 + c_1\theta _7 + c_2 \theta_8]_{W_9}$
                &  $L_N=9$   \\ \hline
$(W_9)^6$ & $[W_9]^6: [\theta _7 + c\theta _8]_{W_9}$  &  $L_N=10$   \\ \hline
$(W_9)^7$ & $[W_9]^7: [\pm\theta _8+c\theta_9]_{W_9}$  &  $L_N=11$   \\ \hline
$(W_9)^8$ & $[W_9]^8: [\pm\theta_9]_{W_9}$     &  $L_N=13$   \\ \hline
$(W_9)^9$ & $[W_9]^9: [0]_{W_9}$
                &  $N$ is contained in a smooth Lagrangian submanifold   \\ \hline
    \end{tabular}

\smallskip

\caption{\small Geometric characterization of symplectic classes of  $W_9$ singularity}\label{tabw9-geom}

\end{small}
\end{table}
\end{center}

\begin{proof}[Sketch of the proof of Theorem \ref{geom-cond-w9}]
The conditions on the pair $(\omega, N)$ in the last column of Table \ref{tabw9-geom} are disjoint. It suffices to prove that these conditions in the row of $(W_9)^{i}$, are satisfied for any $N\in (W_9)^{i}$.
\smallskip

We can take the simplest $2$-forms $\omega ^i$ representing
the normal forms $[W_9]^i$ for algebraic restrictions and we can check that the pair $(W_9,\omega=\omega ^i)$ satisfies the condition
in the last column of Table \ref{tabw9-geom}.

\noindent We note that in the case $N = W_9 = (\ref{defw9})$ one has \par\noindent
$ \ell \ = \Span (\partial /\partial x_3), \  \
V = \Span (\partial /\partial x_2,\partial /\partial x_3$, \
  $W = \Span (\partial/\partial x_1, \partial / \partial x_2, \partial / \partial x_3)$. \par\noindent  By simple calculation and observation of the Lagrangian tangency orders we obtain that the conditions corresponding to the classes $(W_9)^i$ are satisfied.
\end{proof}

\bigskip

\section{Proofs}\label{proofs}

\subsection{The method of algebraic restrictions}
\label{method}

In this section we present basic facts on the method of algebraic restrictions, which is a very powerful tool for the symplectic classification. The details of the method and proofs of all results of this section can be found in \cite{DJZ2}.

  Given a germ of a non-singular manifold $M$ denote by $\Lambda ^p(M)$ the space of all germs at $0$ of differential $p$-forms on $M$. Given a subset $N\subset M$ introduce the following subspaces of $\Lambda ^p(M)$:
$$\Lambda ^p_N(M) = \{\omega \in \Lambda ^p(M): \ \ \omega (x)=0 \ \text {for any} \ x\in N \};$$
$$\mathcal A^p_0(N, M) = \{\alpha  + d\beta : \ \ \alpha \in \Lambda _N^p(M), \ \beta \in \Lambda _N^{p-1}(M).\}$$

\smallskip

\begin{defn}
\label{main-def} Let $N$ be the germ of a subset of $M$ and let
$\omega \in \Lambda ^p(M)$. The {\bf algebraic restriction} of
$\omega $ to $N$ is the equivalence class of $\omega $ in $\Lambda
^p(M)$, where the equivalence is as follows: $\omega $ is
equivalent to $\widetilde \omega $ if $\omega - \widetilde \omega
\in \mathcal A^p_0(N, M)$.
\end{defn}

\noindent {\bf Notation}. The algebraic restriction of the germ of
a $p$-form $\omega $ on $M$ to the germ of a subset $N\subset M$
will be denoted by $[\omega ]_N$. Writing $[\omega ]_N=0$ (or
saying that $\omega $ has zero algebraic restriction to $N$) we
mean that $[\omega ]_N = [0]_N$, i.e. $\omega \in A^p_0(N, M)$.

\medskip




\begin{defn}Two algebraic restrictions
$[\omega ]_N$ and $[\widetilde \omega ]_{\widetilde N}$ are called {\bf
diffeomorphic} if there exists the germ of a diffeomorphism $\Phi:
\widetilde M\to M$ such that $\Phi(\widetilde N)=N$ and  $\Phi ^*([\omega ]_N) =[\widetilde \omega ]_{\widetilde N}$.
\end{defn}

\smallskip


The method of algebraic restrictions applied to singular
quasi-homogeneous subsets is based on the following theorem.

\begin{thm}[Theorem A in \cite{DJZ2}] \label{thm A}
Let $N$ be the germ of a quasi-homogeneous subset of $\mathbb
R^{2n}$. Let $\omega _0, \omega _1$ be germs of symplectic forms
on $\mathbb R^{2n}$ with the same algebraic restriction to $N$.
There exists a local diffeomorphism $\Phi $ such that $\Phi (x) =
x$ for any $x\in N$ and $\Phi ^*\omega _1 = \omega _0$.

Two germs of quasi-homogeneous subsets $N_1, N_2$ of a fixed
symplectic space $(\mathbb R^{2n}, \omega )$ are symplectically
equivalent if and only if the algebraic restrictions of the
symplectic form $\omega $ to $N_1$ and $N_2$ are diffeomorphic.

\end{thm}

\medskip

Theorem \ref{thm A} reduces the problem of symplectic
classification of germs of singular quasi-homogeneous subsets to
the problem of diffeomorphic classification of algebraic
restrictions of the germ of the symplectic form to the germs of
singular quasi-homogeneous subsets.

\smallskip

The geometric meaning of the zero algebraic restriction is explained
by the following theorem.

\begin{thm}[Theorem {\bf B} in \cite{DJZ2}] \label{thm B}  {\it The germ of a quasi-homogeneous
 set $N$  of a symplectic space
$(\mathbb R^{2n}, \omega )$ is contained in a non-singular
Lagrangian submanifold if and only if the symplectic form $\omega
$ has zero algebraic restriction to $N$.}\
\end{thm}

In the paper we use the following notations:

\smallskip

\noindent $\bullet$  $\algrestall $: \ \ the vector space
consisting of  the algebraic restrictions of germs of all $2$-forms on
$\mathbb R^{2n}$ to the germ of a subset $N\subset \mathbb
R^{2n}$;

\smallskip

\noindent $\bullet$  $\algrestclosed$: \ \ the subspace of
$\algrestall $ consisting of the algebraic restrictions of germs of
all closed $2$-forms on $\mathbb R^{2n}$ to $N$;

\smallskip

\noindent $\bullet$ $\algrest $: \ \ the open set in
$\algrestclosed$ consisting of the algebraic restrictions of germs of
all symplectic $2$-forms on $\mathbb R^{2n}$ to $N$.

\medskip

To obtain a classification of the algebraic restrictions we use the following proposition.
\begin{prop}\label{elimin1}
Let $a_1, \cdots, a_p$ be a quasi-homogeneous basis of
quasi-degrees
$\delta_1\le\cdots\le\delta_s<\delta_{s+1}\le\cdots\le\delta_p$ of
the space of algebraic restrictions of closed $2$-forms to $N$.
Let $a=\sum_{j=s}^p c_ja_j$, where $c_j\in \mathbb R$ for
$j=s,\cdots, p$ and $c_s\ne 0$.

If there exists a tangent quasi-homogeneous vector field $X$ over
$N$ such that $\mathcal L_Xa_s=ra_k$ for $k>s$ and $r\ne 0$ then
$a$ is diffeomorphic to $\sum_{j=s}^{k-1} c_ja_j+\sum_{j=k+1}^p
b_j a_j$, for some $b_j\in \mathbb R, \ j=k+1,\cdots,p$.
\end{prop}

Proposition \ref{elimin1} is a modification of  Theorem 6.13 formulated and proved in  \cite{D}. It was formulated for algebraic restrictions to a parameterized curve but we can generalize this theorem for any subset $N$. The proofs of the cited theorem and Proposition \ref{elimin1} are similar and are based on the Moser homotopy method.

\medskip

For calculating discrete invariants we use the following propositions.

\begin{prop}[\cite{DJZ2}]\label{sm}
The symplectic multiplicity  of the germ of a quasi-homogeneous subset $N$ in a symplectic space is equal to the codimension of the orbit of the algebraic restriction $[\omega ]_N$ with respect to the group of local diffeomorphisms preserving $N$  in the space of algebraic restrictions of closed  $2$-forms to $N$.
\end{prop}

\begin{prop}[\cite{DJZ2}]\label{ii}
The index of isotropy  of the germ of a quasi-homogeneous subset $N$ in a symplectic space $(\mathbb R^{2n}, \omega )$ is equal to the maximal order of vanishing of closed $2$-forms representing the algebraic restriction $[\omega ]_N$.
\end{prop}

\begin{prop}[\cite{D}]\label{lto}
Let $f$ be the germ of a quasi-homogeneous curve such that the algebraic restriction of a symplectic form to it can be represented by a closed $2$-form vanishing at $0$. Then the Lagrangian tangency order of the germ of a quasi-homogeneous curve $f$ is the maximum of the order of vanishing on $f$ over all $1$-forms $\alpha$ such that $[\omega]_f=[d\alpha]_f$
\end{prop}

\subsection{Proofs for  $W_8$ singularity}
\subsubsection{Algebraic restrictions to  $W_8$ and their classification}\label{w8-class}

One has the following relations for $(W_8)$-singularities
\begin{equation}
[d(x_2^2+x_1x_3)]_{W_8}=[2x_2dx_2+ x_1dx_3+x_3dx_1]_{W_8}=0,
\label{w81}
\end{equation}
\begin{equation}
[d(x_1^2+x_3^3)]_{W_8}=[2x_1dx_1+3x_3^2dx_3]_{W_8}=0.
\label{w82}
\end{equation}
Multiplying these relations by suitable $1$-forms we obtain the relations in Table \ref{tabw81}.
\renewcommand*{\arraystretch}{1.1}
\begin{footnotesize}
\begin{table}[h]
\begin{center}
\begin{tabular}{|c|p{6.2cm}|p{4.5cm}|}
\hline
       $\delta$ & Relations & Proof\\ \hline

   $14$ & $[x_2dx_2 \wedge dx_3]_N=-\frac{1}{2}[x_3dx_1 \wedge dx_3]_N$
                & (\ref{w81})$\wedge\, dx_3$ \\ \hline

   $15$ &  $[x_1dx_2 \wedge dx_3]_N=[x_3dx_1 \wedge dx_2]_N$
                & (\ref{w81})$\wedge\, dx_2$ \\ \hline
   $16$ & $[x_2dx_1\wedge dx_2]_N=-\frac{1}{2}[x_1dx_1\wedge dx_3]_N=0$ &  (\ref{w82})$\wedge\, dx_3$ and (\ref{w81})$\wedge\, dx_1$ \\ \hline

    $17$ &  $[x_3^2dx_2 \wedge dx_3]_N=\frac{2}{3}[x_1dx_1\wedge dx_2]_N$  & (\ref{w82})$\wedge\, dx_2$ \\ \hline

   $18$ & $[x_3^2dx_1 \wedge dx_3]_N=2[x_2x_3dx_2\wedge dx_3]_N=0$ &  (\ref{w82})$\wedge\, dx_1$ and (\ref{w81})$\wedge\, x_3dx_3$ \\ \hline

   $19$ &  $[x_2^2dx_2\wedge dx_3]_N=-\frac{1}{2}[x_2x_3dx_1\wedge dx_3]_N$
      & (\ref{w81})$\wedge\, x_2dx_3$ \\
      & $[x_2^2dx_2\wedge dx_3]_N=-[x_1x_3dx_2\wedge dx_3]_N=$ \newline \ \ \ \ \ \ \ $=-[x_3^2dx_1 \wedge dx_2]_N$ &  (\ref{w81})$\wedge\, x_3dx_2$  \newline and $[x_2^2+x_1x_3]_N=0$ \\  \hline

   $20$ &   $[\alpha]_N=0$ for all 2-forms $\alpha$ of quasi-degree $20$
                & relations for $\delta\in\{14,15,16\}$ \newline and $[x_2^2+x_1x_3]_N=0$ \\   \hline

   $21$ &   $[\alpha]_N=0$ for all 2-forms $\alpha$ of quasi-degree $21$
                & relations for $\delta\in\{15,16,17\}$ \newline and $[x_1^2+x_3^3]_N=[x_2^2+x_1x_3]_N=0$\\   \hline

   $22$ &   $[\alpha]_N=0$ for all 2-forms $\alpha$ of quasi-degree $22$
                & relations for $\delta\in\{16,17,18\}$ \newline and $[x_1^2+x_3^3]_N=[x_2^2+x_1x_3]_N=0$\\   \hline

   $23$ &   $[\alpha]_N=0$ for all 2-forms $\alpha$ of quasi-degree $23$
                & relations for $\delta\in\{17,18,19\}$ \newline and $[x_1^2+x_3^3]_N=0$ \\   \hline
   $24$ &   $[\alpha]_N=0$ for all 2-forms $\alpha$ of quasi-degree $24$
                & relations for $\delta\in\{18,19,20\}$ \newline and $[x_1^2+x_3^3]_N=0$ \\   \hline
   $25$ &   $[\alpha]_N=0$ for all 2-forms $\alpha$ of quasi-degree $25$
                & relations for $\delta\in\{19,20,21\}$  \\   \hline

   $\delta\!>\!25$ & $[\alpha]_N=0$ for all 2-forms $\alpha$ of quasi-degree $\delta\!>\!25$
                &  relations for $\delta> 19$ \\   \hline

\end{tabular}
\end{center}
\smallskip
\caption{\small Relations towards calculating $[\Lambda^2(\mathbb R^{2n})]_N$ for $N=W_8$}\label{tabw81}
\end{table}
\end{footnotesize}

 Using the method of algebraic restrictions and Table \ref{tabw81} we obtain the following proposition:

\begin{prop}
\label{w8-all}
The space $[\Lambda ^{2}(\mathbb R^{2n})]_{W_8}$ is a $9$-dimensional vector space spanned by the algebraic restrictions to $W_8$ of the $2$-forms

\smallskip

$\theta _1= dx_2\wedge dx_3, \;\; \theta _2=dx_1\wedge dx_3,\;\; \theta_3 = dx_1\wedge dx_2,$

\smallskip

  $\theta _4 = x_3dx_2\wedge dx_3,\;\; \theta _5 = x_2dx_2\wedge dx_3,$\;\; $\sigma_1 = x_1 dx_2\wedge dx_3,$ \ \ $\sigma_2 = x_2 dx_1\wedge dx_3,$

\smallskip

$\theta _7= x_3^2 dx_2\wedge dx_3$,\;\; $\theta _8= x_2^2 dx_2\wedge dx_3$.
\end{prop}

Proposition \ref{w8-all} and results of Section \ref{method}  imply the following description of the space $[Z ^2(\mathbb R^{2n})]_{W_8}$ and the manifold $[{\rm Symp} (\mathbb R^{2n})]_{W_8}$.

\begin{thm} \label{w8-baza}
 The space $[Z^2(\mathbb R^{2n})]_{W_8}$ is a $8$-dimensional vector space
 spanned by the algebraic restrictions to $W_8$  of the quasi-homogeneous $2$-forms $\theta_i$  of degree $\delta_i$

\smallskip

$\theta _1= dx_2\wedge dx_3,\;\;\;\delta_1=9,$

\smallskip

$\theta _2=dx_1\wedge dx_3,\;\;\;\delta_2=10,$

\smallskip

$\theta_3 = dx_1\wedge dx_2,\;\;\;\delta_3=11,$

\smallskip

  $\theta _4 = x_3dx_2\wedge dx_3,\;\;\;\delta_4=13,$

  \smallskip

  $\theta _5 = x_2dx_2\wedge dx_3,\;\;\;\delta_5=14,$

\smallskip

$\theta _6=\sigma_1+\sigma_2=x_1dx_2\wedge dx_3+x_2dx_1\wedge dx_3,\;\;\;\delta_6=15, $

\smallskip

$\theta _7= x_3^2 dx_2\wedge dx_3,\;\;\;\delta_7=17$,

\smallskip

$\theta _8= x_2^2 dx_2\wedge dx_3,\;\;\;\delta_8=19$.

\smallskip

If $n\ge 3$ then $[{\rm Symp} (\mathbb R^{2n})]_{W_8} = [Z^2(\mathbb R^{2n})]_{W_8}$. The manifold $[{\rm Symp} (\mathbb R^{4})]_{W_8}$ is an open part of the $8$-space $[Z^2 (\mathbb R^{4})]_{W_8}$ consisting of algebraic restrictions of the form $[c_1\theta _1 + \cdots + c_8\theta _8]_{W_8}$ such that $(c_1,c_2,c_3)\ne (0,0,0)$.
\end{thm}

\begin{thm}
\label{klasw8} $ \ $

\smallskip

\noindent (i) \ Any algebraic restriction in $[Z ^2(\mathbb R^{2n})]_{W_8}$ can be brought by a symmetry of $W_8$ to one of the normal forms $[W_8]^i$ given in the second column of Table \ref{tabw8}.

\smallskip

\noindent (ii)  \ The codimension in $[Z ^2(\mathbb R^{2n})]_{W_8}$ of the singularity class corresponding to the normal form $[W_8]^i$ is equal to $i$, the symplectic multiplicity and the index of isotropy are given in the fourth and fifth columns of Table  \ref{tabw8}.

\smallskip

\noindent (iii) \ The singularity classes corresponding to the normal forms are disjoint.

\smallskip

\noindent (iv) \ The parameters $c, c_1, c_2$ of the normal forms $[W_8]^i$ are moduli.

\end{thm}

\renewcommand*{\arraystretch}{1.2}
\begin{center}
\begin{table}[h]

    \begin{small}
    \noindent
    \begin{tabular}{|p{3cm}|p{5.8cm}|c|c|c|}
                      \hline
    Symplectic class &   Normal forms for algebraic restrictions    & cod & $\mu ^{\rm sym}$ &  ind  \\ \hline
  $(W_8)^0$ \;\;  $(2n\ge 4)$ & $[W_8]^0: [\theta _1 + c_1\theta _2 + c_2\theta _3]_{W_8}$,\;
                          &  $0$ & $2$ & $0$  \\  \hline
  $(W_8)^1$ \;\; $(2n\ge 4)$ & $[W_8]^1: [c_1\theta _2 + \theta _3 + c_2\theta _4]_{W_8}, \ c_1\ne 0$
                                         &  $1$ & $3$ & $0$ \\ \hline
    $(W_8)^{2,a}$ \; $(2n\ge 4)$& $[W_8]^{2,a}: [\pm\theta _2 + c_1\theta_4+c_2\theta_7]_{W_8}$,
          & $2$ & $4$ & $0$     \\ \hline

    $(W_8)^{2,b}$ \; $(2n\ge 4)$& $[W_8]^{2,b}: [\theta _3 +c_1\theta_4+c_2\theta_5 ]_{W_8}$,
               & $2$ & $4$ & $0$     \\ \hline

    $(W_8)^3$ \;\; $(2n\ge 6)$ & $[W_8]^3: [\theta _4 + c_1\theta _5+c_2\theta _6]_{W_8}$
                                         &  $3$ & $5$ & $1$    \\ \hline
    $(W_8)^4$ \;\; $(2n\ge 6)$ & $[W_8]^4: [\pm\theta _5 + c_1\theta _6 + c_2 \theta_7]_{W_8}$
                                         &  $4$ & $6$ & $1$ \\ \hline
    $(W_8)^5$ \;\; $(2n\ge 6)$ & $[W_8]^5: [\theta _6 + c\theta _7]_{W_8}$
                                         &  $5$ & $6$ & $1$ \\ \hline
    $(W_8)^6$ \;\; $(2n\ge 6)$ & $[W_8]^6: [\theta _7+c\theta_8]_{W_8}$
                                         &  $6$ & $7$ & $2$   \\ \hline
    $(W_8)^7$ \;\; $(2n\ge 6)$ & $[W_8]^7: [\theta_8]_{W_8}$
                                         &  $7$ & $7$ & $2$   \\ \hline
    $(W_8)^8$ \;\; $(2n\ge 6)$ & $[W_8]^8: [0]_{W_8}$ &  $8$ & $8$ & $\infty $ \\ \hline
\end{tabular}

\smallskip

\caption{\small Classification of symplectic $W_8$ singularities:  \newline
$cod$ -- codimension of the classes; \ $\mu ^{sym}$-- symplectic multiplicity; \newline $ind$ -- the index of isotropy.}\label{tabw8}

\end{small}
\end{table}
\end{center}

\medskip

\noindent In the first column of Table \ref{tabw8}  by $(W_8)^i$ we denote a subclass of $(W_8)$ consisting of $N\in (W_8)$ such that the algebraic restriction $[\omega ]_N$ is diffeomorphic to some algebraic restriction of the normal form $[W_8]^i$, where $i$ is the codimension of the class. Classes $(W_8)^{2a}$ and $(W_8)^{2b}$ have the same codimension equal to $2$ but they can be distinguished geometrically (see Table \ref{tabw8-geom}).

The proof of Theorem \ref{klasw8} is presented in Section \ref{w8-proof}.

\subsubsection{Symplectic normal forms. Proof of Theorem \ref{w8-main}}
\label{w8-normal}

Let us transfer the normal forms $[W_8]^i$  to symplectic normal forms
. Fix a family $\omega ^i$ of symplectic forms on $\mathbb R^{2n}$ realizing the family $[W_8]^i$ of algebraic restrictions.  We can fix, for example

\smallskip
\begin{small}

\noindent $\omega ^0 = \theta _1 + c_1\theta _2 + c_2\theta _3 +
dx_1\wedge dx_4 + dx_5\wedge dx_6 + \cdots + dx_{2n-1}\wedge
dx_{2n};$

\smallskip

\noindent $\omega ^1 = c_1 \theta _2 + \theta _3 + c_2\theta _4  +
dx_3\wedge dx_4 + dx_5\wedge dx_6 + \cdots + dx_{2n-1}\wedge
dx_{2n}, \ \ c_1\ne 0; $

\smallskip

\noindent $\omega ^{2,a} = \pm\theta_2+ c_1 \theta _4 + c_2 \theta _7 + dx_2\wedge dx_4 +
dx_5\wedge dx_6 + \cdots + dx_{2n-1}\wedge dx_{2n};$

\smallskip

\noindent $\omega ^{2,b} = \theta_3+ c_1 \theta _4 + c_2 \theta _5 + dx_3\wedge dx_4 +
dx_5\wedge dx_6 + \cdots + dx_{2n-1}\wedge dx_{2n};$

\smallskip

\noindent $\omega ^3 = \theta _4 + c_1\theta _5 + c_2\theta _6+ dx_1\wedge dx_4 +
dx_2\wedge dx_5 + dx_3\wedge dx_6 + dx_7\wedge dx_8 + \cdots +
dx_{2n-1}\wedge dx_{2n};$

\smallskip

\noindent $\omega ^4 = \pm \theta _5 + c_1\theta _6 + c_2\theta _7+ dx_1\wedge dx_4 +dx_2\wedge dx_5 +
dx_3\wedge dx_6 + dx_7\wedge dx_8 + \cdots + dx_{2n-1}\wedge dx_{2n};$

\smallskip

\noindent $\omega ^5 = \theta _6 + c\theta _7+ dx_1\wedge dx_4 + dx_2\wedge dx_5 + dx_3\wedge dx_6 + dx_7\wedge dx_8 + \cdots + dx_{2n-1}\wedge
dx_{2n};$

\smallskip

\noindent $\omega ^6 = \theta _7 + c\theta _8+ dx_1\wedge dx_4 + dx_2\wedge dx_5 +
dx_3\wedge dx_6 + dx_7\wedge dx_8 + \cdots + dx_{2n-1}\wedge
dx_{2n};$

\smallskip

\noindent $\omega ^7 = \theta _8+ dx_1\wedge dx_4 + dx_2\wedge dx_5 +
dx_3\wedge dx_6 + dx_7\wedge dx_8 + \cdots + dx_{2n-1}\wedge
dx_{2n};$

\smallskip

\noindent $\omega ^8 = dx_1\wedge dx_4 + dx_2\wedge dx_5 +
dx_3\wedge dx_6 + dx_7\wedge dx_8 + \cdots + dx_{2n-1}\wedge
dx_{2n}.$
\end{small}

\medskip
 Let $\omega_0 = \sum_{i=1}^m dp_i \wedge dq_i$, where $(p_1,q_1,\cdots,p_n,q_n)$ is the coordinate system on $\mathbb R^{2n}, n\ge 3$ (resp. $n=2$). Fix, for $i=0,1,\cdots ,8$ (resp. for $i = 0,1,2)$ a family $\Phi ^i$ of local diffeomorphisms which bring the family of symplectic forms $\omega ^i$ to the symplectic form $\omega_0$: $(\Phi ^i)^*\omega ^i = \omega_0$. Consider the families $W_8^i = (\Phi ^i)^{-1}(W_8)$. Any stratified submanifold of the symplectic space $(\mathbb R^{2n},\omega_0)$ which is diffeomorphic to $W_8$ is symplectically equivalent to one and only one of the normal forms $W_8^i, i = 0,1,\cdots ,8$ (resp. $i= 0,1,2$) presented in Theorem \ref{w8-main}. By Theorem \ref{klasw8} we obtain that  parameters $c,c_1,c_2$ of the normal forms are moduli.

 \medskip

\subsubsection{Proof of Theorem \ref{klasw8}}
\label{w8-proof}

In our proof we use vector fields tangent to $N\in W_8$. Any vector fields tangent to $N\in W_8$ can be described as $V=g_1E+g_2\mathcal{H}$ where $E$ is the Euler vector field and $\mathcal{H}$ is a Hamiltonian vector field and $g_1, g_2$ are functions. It was shown in \cite{DT1} (Prop. 6.13) that the action of a Hamiltonian vector field on any 1-dimensional complete intersection is trivial.

\medskip

\noindent The germ of a vector field tangent to $W_8$ of non trivial action on algebraic restrictions of closed 2-forms to  $W_8$ may be described as a linear combination germs of vector fields:  $X_0\!=\!E, \  X_1\!=\!x_3E, \ X_2\!=x_2E, \ X_3\!=x_1E, \ X_4\!=x_3^2E, \ X_5\!=x_2x_3E$, \ $X_6\!=x_2^2E, \ X_7\!=x_1x_3E$, where $E$ is the Euler vector field \par\noindent $E= 6 x_1 \partial /\partial x_1+5 x_2 \partial /\partial x_2+4 x_3 \partial /\partial x_3$.

\begin{prop} \label{w8-infinitesimal}

The infinitesimal action of germs of quasi-homogeneous vector
fields tangent to $N\in (W_8)$ on the basis of the vector space of
algebraic restrictions of closed $2$-forms to $N$ is presented in
Table \ref{infini-w8}.

\renewcommand*{\arraystretch}{1.2}
\begin{small}
\begin{table}[h]
\begin{center}
\begin{tabular}{|l|r|r|r|r|r|r|r|r|}

 \hline

  $\mathcal L_{X_i} [\theta_j]$ & $[\theta_1]$   &   $[\theta_2]$ &   $[\theta_3]$ & $[\theta_4]$   & $[\theta_5]$   & $[\theta_6]$   & $[\theta_7]$ & $[\theta_8]$  \\ \hline

  $X_0\!=\!E$ & $9 [\theta_1]$ & $10 [\theta_2]$ & $11 [\theta_3]$ & $13 [\theta_4]$ & $14 [\theta_5]$ & $15 [\theta_6]$ & $17 [\theta_7]$  & $19 [\theta_8]$  \\ \hline

  $X_1\!=\!x_3E$  & $13[\theta_4]$ & $-28 [\theta_5]$ & $5[\theta_6]$  & $17[\theta_7]$  & $[0]$ & $-57[\theta_8]$  & $[0]$ & $[0]$  \\  \hline

  $X_2\!=\!x_2E$ & $14[\theta_5]$ & $ 10[\theta_6]$ & $[0]$ & $[0]$ & $19[\theta_8] $ & $[0]$ & $[0]$ & $[0]$  \\  \hline

  $X_3\!=\!x_1 E$ & $5 [\theta_6]$  & $[0]$ &$\frac{51}{2}[\theta_7]$   & $-19[\theta_8]$ & $[0]$ & $[0]$  & $[0]$  & $[0]$  \\  \hline

  $X_4\!=\!x_3^2E$ & $ 17[\theta_7]$ & $[0]$ & $-19[\theta_8]$ & $[0]$ & $[0]$ & $[0]$ & $[0]$ & $[0]$  \\ \hline

  $X_5\!=\!x_2x_3 E$ & $[0]$ & $-38[\theta_8]$  & $[0]$  & $[0]$ & $[0]$ & $[0]$  & $[0]$  & $[0]$  \\  \hline

  $X_6\!=\!x_2^2 E$ & $19[\theta_8]$ & $[0]$ & $[0]$ & $[0]$ & $[0]$ & $[0]$ & $[0]$ & $[0]$ \\ \hline

  $X_7\!=\!x_1x_3 E$ & $-19[\theta_8]$ & $[0]$  & $[0]$ & $[0]$ & $[0]$ & $[0]$ & $[0]$ &  $[0]$ \\ \hline

  \end{tabular}
\end{center}
\smallskip

\caption{\small Infinitesimal actions on algebraic restrictions of closed
2-forms to   $W_8$. $E=6x_1 \partial /\partial x_1+ 5x_2 \partial /\partial x_2+ 4x_3 \partial /\partial x_3$}\label{infini-w8}
\end{table}
\end{small}
\end{prop}

\newpage

\medskip

Let $\mathcal{A}=[c_1 \theta_1+c_2 \theta_2+c_3 \theta_3+c_4 \theta_4+c_5 \theta_5+c_6 \theta_6 +c_7 \theta_7 +c_8 \theta_8]_{W_8}$
be the algebraic restriction of a symplectic form $\omega$.

The first statement of Theorem \ref{klasw8} follows from the following lemmas.

\begin{lem}
\label{w8lem0} If \;$c_1\ne 0$\; then the algebraic restriction
$\mathcal{A}=[\sum_{k=1}^8 c_k \theta_k]_{W_8}$
can be reduced by a symmetry of $W_8$ to an algebraic restriction $[\theta_1+\widetilde{c}_2 \theta_2+\widetilde{c}_3 \theta_3]_{W_8}$.
\end{lem}

\begin{proof}
Using the data of Table \ref{infini-w8} we can see that for any algebraic restriction $[\theta_k]_{W_8}$, where
 $k\!\in\!\{4,5,\ldots,8\}$ we can find a vector field  $V_k$ tangent to $W_8$ such that $\mathcal L_{V_k}[\theta_1]_{W_8}\!=\![\theta_k]_{W_8}$. We deduce from Proposition \ref{elimin1} that the algebraic restriction  $\mathcal{A}$ is diffeomorphic to  $[c_1\theta_1+{c}_2 \theta_2+{c}_3 \theta_3]_{W_8}$.

By the condition $c_1\ne 0$ we have a diffeomorphism  $\Psi \in Symm(W_8)$ of the form
  \begin{equation}
\label{proofw8lem04}
\Psi:\,(x_1,x_2,x_3)\mapsto (c_1^{-\frac{6}{9}} x_1,c_1^{-\frac{5}{9}} x_2,c_1^{-\frac{4}{9}} x_3)
\end{equation}
and finally we obtain
\[ \Psi^*([c_1\theta_1+{c}_2 \theta_2+{c}_3 \theta_3]_{W_8})=[ \theta_1+c_2 c_1^{-\frac{10}{9}} \theta_2+c_3 c_1^{-\frac{11}{9}} \theta_3]_{W_8} =
 [ \theta_1+ \widetilde{c}_2 \theta_2+\widetilde{c}_3 \theta_3]_{W_8}.\]

\end{proof}

\begin{lem}
\label{w8lem1} If \;$c_1\!=\!0$\;and $c_2\cdot c_3\!\ne\!0$ then the algebraic restriction  $\mathcal{A}$ 
can be reduced by a symmetry of $W_8$ to an algebraic restriction $[\widetilde{c}_2 \theta_2+ \theta_3+\widetilde{c}_4 \theta_4]_{W_8}$.
\end{lem}

\begin{proof}[Proof of Lemma \ref{w8lem1}]

We use the homotopy method to prove that  $\mathcal{A}$ is diffeomorphic to $[\widetilde{c}_2 \theta_2+ \theta_3+\widetilde{c}_4 \theta_4]_{W_8}$.

Let $\mathcal{B}_t=[c_2 \theta_2+c_4 \theta_3+c_4 \theta_4+(1-t)c_5 \theta_5+(1-t)c_6 \theta_6+(1-t)c_7 \theta_7 +(1-t)c_8 \theta_8]_{W_8}$
\; for $t \in[0;1]$. Then $\mathcal{B}_0=\mathcal{A}$\; and \;$\mathcal{B}_1=[c_2 \theta_2+c_3 \theta_3+c_4 \theta_4]_{W_8}$.
 We prove that there exists a family $\Phi_t \in Symm(W_8),\;t\in [0;1]$ such that
 \begin{equation}
\label{proofw8lem11}   \Phi_t^*\mathcal{B}_t=\mathcal{B}_0,\;\Phi_0=id.
\end{equation}
Let $V_t$ be a vector field defined by $\frac{d \Phi_t}{dt}=V_t(\Phi_t)$. Then differentiating (\ref{proofw8lem11}) we obtain
 \begin{equation}
\label{proofw8lem12}   \mathcal L_{V_t} \mathcal{B}_t=[c_5 \theta_5+c_6 \theta_6+c_7 \theta_7 +c_8 \theta_8]_{W_8}.
\end{equation}
We are looking for $V_t$ in the form $V_t=\sum_{k=1}^5 b_k(t) X_k$   where the $b_k(t)$ for $k=1,\ldots,5$ are smooth functions $b_k:[0;1]\rightarrow \mathbb{R}$. Then by Proposition  \ref{w8-infinitesimal} equation (\ref{proofw8lem12}) has a form
\begin{equation}  \label{proofw8lem13}
\left[ \begin{array}{ccccc}
-28c_2 & 0 & 0 & 0 & 0 \\
5c_3 & 10c_2 & 0 & 0 & 0 \\
17c_4 & 0 & \frac{51}{2}c_3 & 0 & 0 \\
-57c_6(1-t) & 19c_5(1-t) & -19c_4 & -19c_3 &-38c_2
\end{array} \right]
\left[ \begin{array}{c} b_1 \\ b_2 \\ b_3 \\ b_4 \\ b_5  \end{array} \right] =
\left[ \begin{array}{c} c_5 \\ c_6 \\ c_7 \\ c_8  \end{array}  \right]
\end{equation}
\noindent If \;$c_2\cdot c_3\ne 0$ we can solve (\ref{proofw8lem13}) and $\Phi_t$ may be obtained as a flow of the vector field $V_t$.
The family $\Phi_t$ preserves $W_8$, because $V_t$ is tangent to $W_8$ and $\Phi_t^*\mathcal{B}_t=\mathcal{A}$.
Using the homotopy arguments we have  $\mathcal{A}$  diffeomorphic to $\mathcal{B}_1=[c_2 \theta_2+c_3 \theta_3+c_4 \theta_4]_{W_8}$.
By the condition $c_3\ne 0$ we have a diffeomorphism $\Psi \in Symm(W_8)$ of the form
  \begin{equation}
\label{proofw8lem14}
\Psi:\,(x_1,x_2,x_3)\mapsto (c_3^{-\frac{6}{11}} x_1,c_3^{-\frac{5}{11}} x_2,c_3^{-\frac{4}{11}} x_3),
\end{equation}
and we obtain \ 
\[ \Psi^*(\mathcal{B}_1)=[c_2 c_3^{-\frac{10}{11}} \theta_2+ \theta_3+c_4 c_3^{-\frac{13}{11}} \theta_3]_{W_8} =
 [ \widetilde{c}_2 \theta_2+ \theta_3+\widetilde{c}_4 \theta_4]_{W_8}.\]
\end{proof}

\begin{lem}
\label{w8lem2a} If $c_1\!=\!c_3\!=\!0$ and $c_2\!\neq\!0$ then the algebraic restriction $\mathcal{A}$
can be reduced by a symmetry of \, $W_8$ to an algebraic restriction $[\pm\theta_2+\widetilde{c}_4 \theta_4+\widetilde{c}_7 \theta_7]_{W_8}$. 
\end{lem}
 \begin{proof}
 We can see from Table \ref{infini-w8} that  for any algebraic restriction  $[\theta_k]_{W_8}$, where $k\in \{5,6,8\}$ there exists a vector field $V_k$ tangent to $W_8$ such that $\mathcal L_{V_k}[\theta_2]_{W_8}=[\theta_k]_{W_8}$. Using Proposition  \ref{elimin1} we obtain that $\mathcal{A}$ is diffeomorphic to  $ [c_2\theta_2+{c}_4 \theta_4+\widehat{c}_7 \theta_7]_{W_8}$ for some $\widehat{c}_7\in\mathbb{R}$.

By the condition $c_2\ne 0$ we can use a diffeomorphism $\Psi \in Symm(W_8)$ of the form
  \begin{equation}
\label{proofw8lem2a4}
\Psi:\,(x_1,x_2,x_3)\mapsto (|c_2|^{-\frac{6}{10}} x_1,|c_2|^{-\frac{5}{10}} x_2,|c_2|^{-\frac{4}{10}} x_3)
\end{equation}
and we obtain
\[\Psi^*([c_2\theta_2+c_4 \theta_4+\widehat{c}_7\theta_7]_{W_8})\!=\![\frac{c_2}{|c_2|}\theta_2+c_4 |c_2|^{-\frac{13}{10}} \theta_4+\widehat{c}_7 |c_2|^{-\frac{17}{10}} \theta_7]_{W_8}\!=\! [\pm\theta_2+\widetilde{c}_4 \theta_4+\widetilde{c}_7\theta_7]_{W_8}.\]

The algebraic restrictions $[ \theta_2+ \widetilde{c}_4 \theta_4+\widetilde{c}_7 \theta_7]_{W_8}$ and $[-\theta_2+ \widetilde{b}_4 \theta_4+\widetilde{b}_7 \theta_7]_{W_8}$ are not diffeomorphic.
Any diffeomorphism $\Phi=(\Phi_1,\ldots,\Phi_{2n})$ of $(\mathbb{R}^{2n},0)$ preserving $W_8$ has to preserve a curve $C(t)=(t^6,t^5,-t^4,0,\ldots,0)$ which means that

$\Phi_1(t^6,t^5,-t^4,0,\ldots,0)=(\psi(t))^6$,\par $\Phi_2(t^6,t^5,-t^4,0,\ldots,0)=(\psi(t))^5$,\par
$\Phi_3(t^6,t^5,-t^4,0,\ldots,0)=-(\psi(t))^4$,\par
$\Phi_k(t^6,t^5,-t^4,0,\ldots,0)=0$ for $k>3$,\par\noindent
 where $\psi(t)=a_1t+a_2t^2+a_3t^3+\ldots$ \; is a diffeomorphism of $(\mathbb{R},0)$. \par Hence $\Phi$ has a linear part
 \setlength{\arraycolsep}{1mm}
 \renewcommand*{\arraystretch}{0.9}
\begin{equation}
\label{lindyfw8}
\begin{array}{cccccccccccccccc}
\Phi_1: & A^6x_1& &+ & & &&A_{14}x_4&+& \cdots &+&A_{1,2n}x_{2n}\\
\Phi_2: & A_{2,1}x_1 &+& A^5x_2& &+& &A_{24}x_4&+& \cdots &+&A_{2,2n}x_{2n} \\
\Phi_3: & A_{3,1}x_1&+ &A_{3,2}x_2 &+ & A^4x_3&+&A_{34}x_4&+& \cdots &+&A_{3,2n}x_{2n}\\
\Phi_4: & & & & & & &A_{44}x_4&+& \cdots &+&A_{4,2n}x_{2n}\\
\vdots & & & & & & & \vdots & \vdots & \vdots & \vdots\\
\Phi_{2n}: & & & & & & &A_{2n,4}x_4&+& \cdots &+&A_{2n,2n}x_{2n},\\
\end{array}
\end{equation}
where $A, A_{i,j}\in \mathbb R$.\par\noindent
If we assume that $\Phi^*([ \theta_2+ \widetilde{c}_4 \theta_4+\widetilde{c}_7 \theta_7]_{W_8})= [-\theta_2+ \widetilde{b}_4 \theta_4+\widetilde{b}_7 \theta_7]_{W_8}$, then \par\noindent
 $A^{10}dx_1\wedge dx_3|_0=-dx_1\wedge dx_3|_0$, which is a contradiction.
\end{proof}

\begin{lem}
\label{w8lem2b} If $c_1\!=\!c_2\!=\!0$ and $c_3\!\neq\!0$ then the algebraic restriction $\mathcal{A}$
can be reduced by a symmetry of \, $W_8$ to an algebraic restriction $[\theta_3+\widetilde{c}_4 \theta_4+\widetilde{c}_5 \theta_5]_{W_8}$.
\end{lem}

\begin{lem}
\label{w8lem3} If \;$c_1=c_2=c_3=0$ and $c_4\ne 0$\; then the algebraic restriction $\mathcal{A}$
can be reduced by a symmetry of \, $W_8$ to an algebraic restriction   $[\theta_4+\widetilde{c}_5 \theta_5+\widetilde{c}_6 \theta_6]_{W_8}$.
\end{lem}

\begin{lem}
\label{w8lem4} If \;$c_1=0,\ldots,c_4=0$ and $c_5\ne 0$,\; then the algebraic restriction $\mathcal{A}$
can be reduced by a symmetry of \, $W_8$ to an algebraic restriction $[\pm\theta_5+\widetilde{c}_6 \theta_6+\widetilde{c}_7 \theta_7]_{W_8}$.
\end{lem}

\begin{lem}
\label{w8lem5} If \;$c_1=0,\ldots,c_5=0$ and $c_6\ne 0$\; then the algebraic restriction $\mathcal{A}$
can be reduced by a symmetry of \, $W_8$ to an algebraic restriction $[\theta_6+\widetilde{c}_7 \theta_7]_{W_8}$.
\end{lem}

\begin{lem}
\label{w8lem6} If \;$c_1=0,\ldots,c_6=0$ and $c_7\ne 0$\; then the algebraic restriction $\mathcal{A}$
can be reduced by a symmetry of \, $W_8$ to an algebraic restriction $[\theta_7+\widetilde{c}_8 \theta_8]_{W_8}$.
\end{lem}

\begin{lem}
\label{w8lem7} If \;$c_1=0,\ldots,c_7=0$ and $c_8\ne 0$\; then the algebraic restriction $\mathcal{A}$
can be reduced by a symmetry of \, $W_8$ to an algebraic restriction  $[\theta_8]_{W_8}$.
\end{lem}

The proofs of Lemmas \ref{w8lem2b} -- \ref{w8lem7} are similar to the proofs of Lemmas \ref{w8lem0} -- \ref{w8lem2a} and are based on Table \ref{infini-w8}.

\medskip

Statement $(ii)$ of Theorem \ref{klasw8} follows from the conditions
in the proof of part $(i)$ (codimension) and  from Theorem \ref{thm B} and Proposition \ref{sm} (symplectic multiplicity) and Proposition \ref{ii} (index of isotropy).

\medskip

To prove statement $(iii)$ of Theorem \ref{klasw8} we have to show that singularity classes corresponding to normal forms are disjoint. The singularity classes that can be distinguished by geometric conditions obviously are disjoint. From Theorem \ref{geom-cond-w8} we see that only classes $(W_8)^1$ and $(W_8)^{2,a}$ can not be distinguished by the geometric conditions.
 To prove that these classes are disjoint we compare  the tangent spaces to the orbits of the respective algebraic restrictions. From Table \ref{infini-w8} we see that the tangent space to the orbit of $[c_1\theta _2 + \theta _3 + c_2\theta _4]_{W_8}$ at $[c_1\theta _2 + \theta _3 + c_2\theta _4]_{W_8}$ is spanned by the linearly independent algebraic restrictions $[10c_2\theta _2 + 11\theta _3 + 13c_2\theta _4]_{W_8}$, $[\theta _5]_{W_8}$, $[\theta _6]_{W_8}$, $[\theta _7]_{W_8}$, $[\theta _8]_{W_8}$ and  the tangent space to the orbit of $[\pm\theta _2 + b_1\theta _4 + b_2\theta _7]_{W_8}$ at $[\pm\theta _2 + b_1\theta _4 + b_2\theta _7]_{W_8}$ is spanned by the linearly independent algebraic restrictions $[\pm 10\theta _2 + 13b_1\theta _4 + 17b_2\theta _7]_{W_8}$, $[\pm 28\theta _5+17b_1\theta_7]_{W_8}$, $[\theta _6]_{W_8}$ and $[\theta _8]_{W_8}$.



\medskip

To prove statement $(iv)$ of Theorem \ref{klasw8} we have to show that the parameters $c, c_1, c_2$ are moduli in the normal forms. The proofs are very similar in all cases. We consider as an example
the normal form with two parameters $[\theta_1+c_1\theta_2+c_2\theta_3]_{W_8}$. From Table \ref{infini-w8} we see that the tangent space to the orbit
of $[\theta_1+c_1\theta_2+c_2\theta_3]_{W_8}$ at $[\theta_1+c_1\theta_2+c_2\theta_3]_{W_8}$ is spanned by the linearly independent algebraic restrictions
$[9\theta_1+10c_1\theta_2+11c_2\theta_3]_{W_8}$, $[\theta_4]_{W_8},[\theta_5]_{W_8}, [\theta_6]_{W_8}, [\theta_7]_{W_8}, [\theta_8]_{W_8}.$ Hence the algebraic restrictions
$[\theta_2]_{W_8}$ and $[\theta_3]_{W_8}$ do not belong to it. Therefore the parameters $c_1$ and $c_2$ are independent moduli in the normal form
$[\theta_1+c_1\theta_2+c_2\theta_3]_{W_8}$.

\bigskip


\subsection{Proofs for  $W_9$ singularity}
\subsubsection{Algebraic restrictions to $W_9$ and their classification}\label{w9-class} $ \ $


One has the following relations for $(W_9)$-singularities
\begin{equation}
[d(x_1^2+x_2x_3^2)]_{W_9}=[2x_1dx_1+ 2x_2x_3dx_3+x_3^2 dx_2]_{W_9}=0,
\label{w91}
\end{equation}
\begin{equation}
[d(x_2^2+x_1x_3)]_{W_9}=[2x_2dx_2+x_3dx_1+x_1dx_3]_{W_9}=0.
\label{w92}
\end{equation}
Multiplying these relations by suitable $1$-forms we obtain the relations in Table \ref{tabw91}.
 \setlength{\tabcolsep}{1mm}
\renewcommand*{\arraystretch}{1.15}
\begin{footnotesize}
\begin{table}[h]
\begin{center}
\begin{tabular}{|c|c|p{3.7cm}|}
\hline
       $\delta$ & Relations & Proof\\ \hline    

$11$ & $[x_3dx_1 \wedge dx_3]_N=-2[x_2dx_2 \wedge dx_3]_N$
                & (\ref{w92})$\wedge\, dx_3$ \\ \hline

$12$ &  $[x_1dx_2 \wedge dx_3]_N=[x_3dx_1 \wedge dx_2]_N$
             & (\ref{w92})$\wedge\, dx_2$  \\ \hline

$13$ &  $[x_3^2dx_2\wedge dx_3]_N=-2[x_1dx_1 \wedge dx_3]_N=4[x_2dx_1 \wedge dx_1]_N$ & (\ref{w91})$\wedge dx_3$ and (\ref{w92})$\wedge dx_1$ \\ \hline

$14$ &  $[x_3^2dx_1\!\wedge dx_3]_N\!=-2[x_1dx_1\!\wedge dx_2]_N\!=\!-2[x_2x_3dx_2\!\wedge dx_3]_N$
             & (\ref{w91})$\wedge dx_2$, \  (\ref{w92})$\wedge x_3dx_3$ \\ \hline

$15$ &  $[\alpha]_N=0$ for all 2-forms $\alpha$ of quasi-degree $15$ & relations for $\delta\in\{11,12\}$ \newline and (\ref{w91})$\wedge dx_1$ \newline and $[x_2^2+x_1x_3]_N=0$\\   \hline

   $16$ &    $[x_3^3dx_2\!\wedge dx_3]_N\!=\!-2[x_1x_3dx_1\!\wedge dx_3]_N\!=4[x_2x_3dx_1\!\wedge dx_2]_N$
                & relations for $\delta=13$  \\
                & $[x_1x_3dx_1\!\wedge dx_3]_N\!=\!-2[x_1x_2dx_2\!\wedge dx_3]_N\!=\!-[x_2^2dx_1\!\wedge dx_3]_N$ & relations for $\delta\in\{11,12\}$ \newline and $[x_2^2+x_1x_3]_N=0$ \\ \hline

   $17$ &   $[\alpha]_N=0$ for all 2-forms $\alpha$ of quasi-degree $17$
                & relations for $\delta\in\{12,13,14\}$ \newline and $[x_1^2+x_2x_3^2]_N=0$
                                    \newline i $[x_2^2+x_1x_3]_N=0$\\   \hline

   $18$ &   $[\alpha]_N=0$ for all 2-forms $\alpha$ of quasi-degree $18$
                & relations for $\delta\in\{13,14,15\}$ \newline and $[x_1^2+x_2x_3^2]_N=0$\\   \hline

  $19$ &   $[\alpha]_N=0$ for all 2-forms $\alpha$ of quasi-degree $19$
                & relations for $\delta\in\{14,15,16\}$ \newline and $[x_1^2+x_2x_3^2]_N=0$\\   \hline

   $20$ &   $[\alpha]_N=0$ for all 2-forms $\alpha$ of quasi-degree $20$
                &  relations for $\delta\in\{15,16,17\}$ \\   \hline

   $21$ &   $[\alpha]_N=0$ for all 2-forms $\alpha$ of quasi-degree $21$
                &  relations for $\delta\in\{16,17,18\}$ \\   \hline

   $>\!21$ &   $[\alpha]_N=0$ for all 2-forms $\alpha$ of quasi-degree  $\delta>21$
                &  relations for $\delta> 16$ \\   \hline

\end{tabular}
\end{center}
\smallskip
\caption{\small Relations towards calculating $[\Lambda^2(\mathbb R^{2n})]_N$ for $N=W_9$}\label{tabw91}
\end{table}
\end{footnotesize}
\smallskip

 Using the method of algebraic restrictions and Table \ref{tabw91} we obtain the following proposition:

\begin{prop}
\label{w9-all}
The space $[\Lambda ^{2}(\mathbb R^{2n})]_{W_9}$ is a $10$-dimensional vector space spanned by the algebraic restrictions to $W_9$ of the $2$-forms

$\theta _1= dx_2\wedge dx_3, \;\; \theta _2=dx_1\wedge dx_3,\;\; \theta_3 = dx_1\wedge dx_2,$


  $\theta _4 = x_3dx_2\wedge dx_3,\;\; \theta _5 = x_3dx_1\wedge dx_3,$\;\; $\sigma_1 = x_1 dx_2\wedge dx_3,$ \ \ $\sigma_2 = x_2 dx_1\wedge dx_3,$


$\theta _7= x_3^2 dx_2\wedge dx_3$,\;\; $\theta _8= x_3^2 dx_1\wedge dx_3$, \;\; $\theta _9= x_3^3 dx_2\wedge dx_3$.
\end{prop}

Proposition \ref{w9-all} and results of Section \ref{method}  imply the following description of the space $[Z ^2(\mathbb R^{2n})]_{W_9}$ and the manifold $[{\rm Symp} (\mathbb R^{2n})]_{W_9}$.

\begin{thm} \label{w9-baza}
 The space $[Z^2(\mathbb R^{2n})]_{W_9}$ is a $9$-dimensional vector space
 spanned by the algebraic restrictions to $W_9$  of the quasi-homogeneous $2$-forms $\theta_i$  of degree $\delta_i$

\smallskip
\begin{small}

$\theta _1= dx_2\wedge dx_3,\;\;\;\delta_1=7,$

\smallskip

$\theta _2=dx_1\wedge dx_3,\;\;\;\delta_2=8,$

\smallskip

$\theta_3 = dx_1\wedge dx_2,\;\;\;\delta_3=9,$

\smallskip

  $\theta _4 = x_3dx_2\wedge dx_3,\;\;\;\delta_4=10,$

  \smallskip

  $\theta _5 = x_3dx_1\wedge dx_3,\;\;\;\delta_5=11,$

\smallskip

$\theta _6=\sigma_1+\sigma_2=x_1dx_2\wedge dx_3+x_2dx_1\wedge dx_3,\;\;\;\delta_6=12, $

\smallskip

$\theta _7= x_3^2 dx_2\wedge dx_3,\;\;\;\delta_7=13$,

\smallskip

$\theta _8= x_3^2 dx_1\wedge dx_3,\;\;\;\delta_8=14$,

\smallskip

$\theta _9= x_3^3 dx_2\wedge dx_3,\;\;\;\delta_8=16$,
\end{small}

\smallskip

If $n\ge 3$ then $[{\rm Symp} (\mathbb R^{2n})]_{W_9} = [Z^2(\mathbb R^{2n})]_{W_9}$. The manifold $[{\rm Symp} (\mathbb R^{4})]_{W_9}$ is an open part of the $9$-space $[Z^2 (\mathbb R^{4})]_{W_9}$ consisting of algebraic restrictions of the form $[c_1\theta _1 + \cdots + c_9\theta _9]_{W_9}$ such that $(c_1,c_2,c_3)\ne (0,0,0)$.
\end{thm}

\begin{thm}
\label{klasw9} $ \ $

\smallskip

\noindent (i) \ Any algebraic restriction in $[Z ^2(\mathbb R^{2n})]_{W_9}$ can be brought by a symmetry of $W_9$ to one of the normal forms $[W_9]^i$ given in the second column of Table \ref{tabw9}.

\smallskip

\noindent (ii)  \ The codimension in $[Z ^2(\mathbb R^{2n})]_{W_9}$ of the singularity class corresponding to the normal form $[W_9]^i$ is equal to $i$, the symplectic multiplicity and the index of isotropy are given in the fourth and fifth columns of Table  \ref{tabw9}.

\smallskip

\noindent (iii) \ The singularity classes corresponding to the normal forms are disjoint.

\smallskip

\noindent (iv) \ The parameters $c, c_1, c_2$ of the normal forms $[W_9]^i$ are moduli.

\end{thm}
 \setlength{\tabcolsep}{1.5mm}
\renewcommand*{\arraystretch}{1.2}
\begin{center}
\begin{table}[h]

    \begin{small}
    \noindent
    \begin{tabular}{|p{3cm}|p{5.8cm}|c|c|c|}
                      \hline
    Symplectic class &   Normal forms for algebraic restrictions    & cod & $\mu ^{\rm sym}$ &  ind  \\ \hline
  $(W_9)^0$ \;\;  $(2n\ge 4)$ & $[W_9]^0: [\theta _1 + c_1\theta _2 + c_2\theta _3]_{W_9}$,\;
                          &  $0$ & $2$ & $0$  \\  \hline
  $(W_9)^1$ \;\; $(2n\ge 4)$ & $[W_9]^1: [\pm\theta _2 + c_1\theta _3 + c_2\theta _4]_{W_9}$
                                         &  $1$ & $3$ & $0$ \\ \hline
    $(W_9)^2$ \;\; $(2n\ge 4)$& $[W_9]^2: [\theta _3 + c_1\theta_4+c_2\theta_5]_{W_9}$,
          & $2$ & $4$ & $0$     \\ \hline

    $(W_9)^3$ \;\; $(2n\ge 6)$& $[W_9]^3: [\pm\theta _4 +c_1\theta_5+c_2\theta_6 ]_{W_9}$,
               & $3$ & $5$ & $1$     \\ \hline

    $(W_9)^4$ \;\; $(2n\ge 6)$ & $[W_9]^4: [\theta _5 + c_1\theta _6+c_2\theta _7]_{W_9}$
                                         &  $4$ & $6$ & $1$    \\ \hline
    $(W_9)^5$ \;\; $(2n\ge 6)$ & $[W_9]^5: [\pm\theta _6 + c_1\theta _7 + c_2 \theta_8]_{W_9}$
                                         &  $5$ & $7$ & $1$ \\ \hline
    $(W_9)^6$ \;\; $(2n\ge 6)$ & $[W_9]^6: [\theta _7 + c\theta _8]_{W_9}$
                                         &  $6$ & $7$ & $2$ \\ \hline
    $(W_9)^7$ \;\; $(2n\ge 6)$ & $[W_9]^7: [\pm\theta _8+c\theta_9]_{W_9}$
                                         &  $7$ & $8$ & $2$   \\ \hline
    $(W_9)^8$ \;\; $(2n\ge 6)$ & $[W_9]^8: [\pm\theta_9]_{W_9}$
                                         &  $8$ & $8$ & $3$   \\ \hline
    $(W_9)^9$ \;\; $(2n\ge 6)$ & $[W_9]^9: [0]_{W_9}$ &  $9$ & $9$ & $\infty $ \\ \hline
\end{tabular}

\smallskip

\caption{\small Classification of symplectic $W_9$ singularities:  \newline
$cod$ -- codimension of the classes; \ $\mu ^{sym}$-- symplectic multiplicity; \newline $ind$ -- the index of isotropy.}\label{tabw9}

\end{small}
\end{table}
\end{center}

\medskip

\noindent In the first column of Table \ref{tabw9}  by $(W_9)^i$ we denote a subclass of $(W_9)$ consisting of $N\in (W_9)$ such that the algebraic restriction $[\omega ]_N$ is diffeomorphic to some algebraic restriction of the normal form $[W_9]^i$.

The proof of Theorem \ref{klasw9} is presented in Section \ref{w9-proof}.

\subsubsection{Symplectic normal forms. Proof of Theorem \ref{w9-main}}
\label{w9-normal} $ \ $

\medskip
\par
Let us transfer the normal forms $[W_9]^i$  to symplectic normal forms. Fix a family $\omega ^i$ of symplectic forms on $\mathbb R^{2n}$ realizing the family $[W_9]^i$ of algebraic restrictions.  We can fix, for example

\smallskip
\begin{small}

\noindent $\omega ^0 = \theta _1 + c_1\theta _2 + c_2\theta _3 +
dx_1\wedge dx_4 + dx_5\wedge dx_6 + \cdots + dx_{2n-1}\wedge
dx_{2n};$

\smallskip

\noindent $\omega ^1 = \pm \theta _2 + c_1\theta _3 + c_2\theta _4  +
dx_2\wedge dx_4 + dx_5\wedge dx_6 + \cdots + dx_{2n-1}\wedge
dx_{2n}; $

\smallskip

\noindent $\omega ^2 = \theta_3+ c_1 \theta _4 + c_2 \theta _5 + dx_3\wedge dx_4 +
dx_5\wedge dx_6 + \cdots + dx_{2n-1}\wedge dx_{2n};$

\smallskip

\noindent $\omega ^3 = \pm\theta _4 + c_1\theta _5 + c_2\theta _6+ dx_1\wedge dx_4 +
dx_2\wedge dx_5 + dx_3\wedge dx_6 + dx_7\wedge dx_8 + \cdots +
dx_{2n-1}\wedge dx_{2n};$

\smallskip

\noindent $\omega ^4 = \theta _5 + c_1\theta _6 + c_2\theta _7+ dx_1\wedge dx_4 +dx_2\wedge dx_5 +
dx_3\wedge dx_6 + dx_7\wedge dx_8 + \cdots + dx_{2n-1}\wedge dx_{2n};$

\smallskip

\noindent $\omega ^5 = \pm\theta _6 + c_1\theta _7+ c_2\theta _8+ dx_1\wedge dx_4 + dx_2\wedge dx_5 + dx_3\wedge dx_6 + dx_7\wedge dx_8 + \cdots + dx_{2n-1}\wedge
dx_{2n};$

\smallskip

\noindent $\omega ^6 = \theta _7 + c\theta _8+ dx_1\wedge dx_4 + dx_2\wedge dx_5 +
dx_3\wedge dx_6 + dx_7\wedge dx_8 + \cdots + dx_{2n-1}\wedge
dx_{2n};$

\smallskip

\noindent $\omega ^7 = \pm\theta _8+ c\theta _9+ dx_1\wedge dx_4 + dx_2\wedge dx_5 +
dx_3\wedge dx_6 + dx_7\wedge dx_8 + \cdots + dx_{2n-1}\wedge
dx_{2n};$

\smallskip

\noindent $\omega ^8 = \pm\theta _9+ dx_1\wedge dx_4 + dx_2\wedge dx_5 +
dx_3\wedge dx_6 + dx_7\wedge dx_8 + \cdots + dx_{2n-1}\wedge
dx_{2n};$

\smallskip

\noindent $\omega ^9 = dx_1\wedge dx_4 + dx_2\wedge dx_5 +
dx_3\wedge dx_6 + dx_7\wedge dx_8 + \cdots + dx_{2n-1}\wedge
dx_{2n}.$
\end{small}

\medskip
 Let $\omega_0 = \sum_{i=1}^m dp_i \wedge dq_i$, where $(p_1,q_1,\cdots,p_n,q_n)$ is the coordinate system on $\mathbb R^{2n}, n\ge 3$ (resp. $n=2$). Fix, for $i=0,1,\cdots ,9$ (resp. for $i = 0,1,2)$ a family $\Phi ^i$ of local diffeomorphisms which bring the family of symplectic forms $\omega ^i$ to the symplectic form $\omega_0$: $(\Phi ^i)^*\omega ^i = \omega_0$. Consider the families $W_9^i = (\Phi ^i)^{-1}(W_8)$. Any stratified submanifold of the symplectic space $(\mathbb R^{2n}, \omega_0)$ which is diffeomorphic to $W_9$ is symplectically equivalent to one and only one of the normal forms $W_9^i, i = 0,1,\cdots ,9$ (resp. $i= 0,1,2$) presented in Theorem \ref{w9-main}. By Theorem \ref{klasw9} we obtain that  parameters $c,c_1,c_2$ of the normal forms are moduli.

\medskip

\subsubsection{Proof of Theorem \ref{klasw9}}
\label{w9-proof} $ \ $

\medskip
In our proof we use vector fields tangent to $N\in W_9$.

\noindent The germ of a vector field tangent to $W_8$ of non trivial action on algebraic restrictions of closed 2-forms to  $W_9$ may be described as a linear combination germs of the following vector fields: \par\noindent $X_0\!=E$, $X_1\!=x_3E, \ X_2\!=x_2E, \ X_3\!=x_1E, \ X_4\!=x_3^2E, X_5\!=x_2x_3E$, \ $X_6\!=x_2^2E$, $X_7\!=x_1x_3E$,  $X_8\!=x_1x_2E$, $X_9\!=x_3^3E,$ \par\noindent where $E$ is the Euler vector field $E= 5 x_1 \partial /\partial x_1+4 x_2 \partial /\partial x_2+3 x_3 \partial /\partial x_3$.

\begin{prop} \label{w9-infinitesimal}

The infinitesimal action of germs of quasi-homogeneous vector
fields tangent to $N\in (W_9)$ on the basis of the vector space of
algebraic restrictions of closed $2$-forms to $N$ is presented in
Table \ref{infini-w9}.

\renewcommand*{\arraystretch}{1.2}
\begin{small}
\begin{table}[h]
\begin{center}
\begin{tabular}{|l|r|r|r|r|r|r|r|r|r|} \hline

  $\mathcal L_{X_i} [\theta_j]$ & $[\theta_1]$   &   $[\theta_2]$ &   $[\theta_3]$ & $[\theta_4]$   & $[\theta_5]$   & $[\theta_6]$   & $[\theta_7]$ & $[\theta_8]$ & $[\theta_9]$ \\ \hline

  $X_0\!=\!E$ & $7 [\theta_1]$ & $8 [\theta_2]$ & $9 [\theta_3]$ & $10 [\theta_4]$ & $11 [\theta_5]$ & $12 [\theta_6]$ & $13 [\theta_7]$  & $14 [\theta_8]$ & $16[\theta_9]$ \\ \hline

  $X_1\!=\!x_3E$  & $10[\theta_4]$ & $11 [\theta_5]$ & $4[\theta_6]$  & $13[\theta_7]$  & $14[\theta_8]$ & $[0]$  & $16[\theta_9]$ & $[0]$ & $[0]$ \\  \hline

  $X_2\!=\!x_2E$ & $-\frac{11}{2} [\theta_5]$ & $ 8[\theta_6]$ & $\frac{13}{4} [\theta_7]$ & $-7 [\theta_8]$ & $ [0]$ & $12[\theta_9]$ & $[0]$ & $[0]$ & $[0]$ \\  \hline

  $X_3\!=\!x_1 E$ & $4 [\theta_6]$  & $-\frac{13}{2}[\theta_7]$ & $-7[\theta_8]$  & $[0]$ & $-8[\theta_9]$ & $[0]$  & $[0]$  & $[0]$ & $[0]$ \\  \hline

  $X_4\!=\!x_3^2E$ & $ 13[\theta_7]$ & $14[\theta_8]$  & $[0]$ & $16[\theta_9]$ & $[0]$ & $[0]$ & $[0]$ & $[0]$ & $[0]$ \\ \hline

  $X_5\!=\!x_2x_3 E$ & $-7[\theta_8]$  & $[0]$ & $4[\theta_9]$  & $[0]$ & $[0]$ & $[0]$  & $[0]$  & $[0]$ & $[0]$ \\  \hline

  $X_6\!=\!x_2^2 E$ & $[0]$  & $8[\theta_9]$  & $[0]$ & $[0]$ & $[0]$ & $[0]$ & $[0]$ & $[0]$ & $[0]$ \\ \hline

  $X_7\!=\!x_1x_3 E$ & $[0]$  & $-8[\theta_9]$  & $[0]$ & $[0]$ & $[0]$ & $[0]$ & $[0]$ & $[0]$ & $[0]$ \\ \hline

  $X_8\!=\!x_1x_2 E$ & $4[\theta_9]$  & $[0]$ & $[0]$  & $[0]$ & $[0]$ & $[0]$  & $[0]$  & $[0]$ & $[0]$ \\  \hline

  $X_9\!=\!x_3^3 E$ & $16[\theta_9]$  & $[0]$ & $[0]$  & $[0]$ & $[0]$ & $[0]$  & $[0]$  & $[0]$ & $[0]$ \\  \hline

\end{tabular}
\end{center}

\smallskip

\caption{\small Infinitesimal actions on algebraic restrictions of closed
2-forms to   $W_9$. $E=5x_1 \partial /\partial x_1+ 4x_2 \partial /\partial x_2+ 3x_3 \partial /\partial x_3$}\label{infini-w9}
\end{table}
\end{small}
\end{prop}

\newpage

Let $\mathcal{A}=[\sum_{k=1}^9 \theta_k]_{W_9}$
be the algebraic restriction of a symplectic form $\omega$.

\smallskip

The first statement of Theorem \ref{klasw9} follows from the following lemmas.

\begin{lem}
\label{w9lem0} If \;$c_1\ne 0$\; then the algebraic restriction
$\mathcal{A}=[\sum_{k=1}^9 c_k \theta_k]_{W_9}$
can be reduced by a symmetry of $W_9$ to an algebraic restriction $[\theta_1+\widetilde{c}_2 \theta_2+\widetilde{c}_3 \theta_3]_{W_9}$.
\end{lem}

\begin{lem}
\label{w9lem1} If \;$c_1\!=\!0$\;and $c_2\ne\!0$ then the algebraic restriction  $\mathcal{A}$ 
can be reduced by a symmetry of $W_9$ to an algebraic restriction $[\pm \theta_2+ \widetilde{c}_3\theta_3+\widetilde{c}_4 \theta_4]_{W_9}$.
\end{lem}

\begin{lem}
\label{w9lem2} If \;$c_1\!=c_2=\!0$\;and $c_3\ne\!0$ then the algebraic restriction  $\mathcal{A}$ 
can be reduced by a symmetry of $W_9$ to an algebraic restriction $[\theta_3+ \widetilde{c}_4\theta_4+\widetilde{c}_5 \theta_5]_{W_9}$.
\end{lem}

\begin{lem}
\label{w9lem3} If \;$c_1\!=c_2=c_3=\!0$\;and $c_4\ne\!0$ then the algebraic restriction  $\mathcal{A}$ 
can be reduced by a symmetry of $W_9$ to an algebraic restriction $[\pm\theta_4+ \widetilde{c}_5\theta_5+\widetilde{c}_6 \theta_6]_{W_9}$.
\end{lem}

\begin{lem}
\label{w9lem4} If \;$c_1\!=\ldots=c_4=\!0$\;and $c_5\ne\!0$ then the algebraic restriction  $\mathcal{A}$ 
can be reduced by a symmetry of $W_9$ to an algebraic restriction $[\theta_5+ \widetilde{c}_6\theta_6+\widetilde{c}_7 \theta_7]_{W_9}$.
\end{lem}

\begin{lem}
\label{w9lem5} If \;$c_1\!=\ldots=c_5=\!0$\;and $c_6\ne\!0$ then the algebraic restriction  $\mathcal{A}$ 
can be reduced by a symmetry of $W_9$ to an algebraic restriction $[\pm\theta_6+ \widetilde{c}_7\theta_7+\widetilde{c}_8 \theta_8]_{W_9}$.
\end{lem}

\begin{lem}
\label{w9lem6} If \;$c_1\!=\ldots=c_6=\!0$\;and $c_7\ne\!0$ then the algebraic restriction  $\mathcal{A}$ 
can be reduced by a symmetry of $W_9$ to an algebraic restriction $[\theta_7+\widetilde{c}_8 \theta_8]_{W_9}$.
\end{lem}

\begin{lem}
\label{w9lem7} If \;$c_1\!=\ldots=c_7=\!0$\;and $c_8\ne\!0$ then the algebraic restriction  $\mathcal{A}$ 
can be reduced by a symmetry of $W_9$ to an algebraic restriction $[\pm\theta_8+\widetilde{c}_9 \theta_9]_{W_9}$.
\end{lem}

\begin{lem}
\label{w9lem8} If \;$c_1\!=\ldots=c_8=\!0$\;and $c_9\ne\!0$ then the algebraic restriction  $\mathcal{A}$ 
can be reduced by a symmetry of $W_9$ to an algebraic restriction $[\pm\theta_9]_{W_9}$.
\end{lem}


The proofs of Lemmas \ref{w9lem0} -- \ref{w9lem8} are similar to the proofs of the lemmas for the $W_8$ singularity. As an example we give the proof of Lemma \ref{w9lem1}.

\begin{proof}[Proof of Lemma \ref{w9lem1}]
 We see from Table \ref{infini-w9} that  for any algebraic restriction  $[\theta_k]_{W_9}$, where $k\in \{5,6,7,8,9\}$ there exists a vector field $V_k$ tangent to $W_9$ such that $\mathcal L_{V_k}[\theta_2]_{W_9}=[\theta_k]_{W_9}$. Using Proposition  \ref{elimin1} we obtain that $\mathcal{A}$ is diffeomorphic to  $[c_2\theta_2+{c}_3 \theta_3+c_4\theta_4]_{W_9}$.

By the condition $c_2\ne 0$ we have a diffeomorphism $\Psi \in Symm(W_9)$ of the form
  \begin{equation}
\label{proofw9lem24}
\Psi:\,(x_1,x_2,x_3)\mapsto (|c_2|^{-\frac{5}{8}} x_1,|c_2|^{-\frac{4}{8}} x_2,|c_2|^{-\frac{3}{8}} x_3)
\end{equation}
and we obtain
\[\Psi^*([c_2\theta_2+c_3 \theta_3+c_4 \theta_4]_{W_9})\!=\![\frac{c_2}{|c_2|}\theta_2+c_3 |c_2|^{-\frac{9}{8}} \theta_3+{c}_4 |c_2|^{-\frac{10}{8}} \theta_4]_{W_9}\!=\![\pm\theta_2+ \widetilde{c}_3 \theta_3+\widetilde{c}_4 \theta_4]_{W_9}.\]

The algebraic restrictions $[ \theta_2+ \widetilde{c}_3 \theta_3+\widetilde{c}_4 \theta_4]_{W_9}$ and $[-\theta_2+ \widetilde{b}_3 \theta_3+\widetilde{b}_4 \theta_4]_{W_9}$ are not diffeomorphic.
Any diffeomorphism $\Phi=(\Phi_1,\ldots,\Phi_{2n})$ of $(\mathbb{R}^{2n},0)$ preserving $W_9$ has to preserve a curve $C_2(t)=(t^5,-t^4,-t^3,0,\ldots,0)$. 
Hence $\Phi$ has a linear part
 \setlength{\arraycolsep}{1mm}
 \renewcommand*{\arraystretch}{0.9}
\begin{equation}
\label{lindyfw9}
\begin{array}{cccccccccccccccc}
\Phi_1: & A^5x_1& &+ & & &&A_{14}x_4&+& \cdots &+&A_{1,2n}x_{2n}\\
\Phi_2: & A_{2,1}x_1 &+& A^4x_2& &+& &A_{24}x_4&+& \cdots &+&A_{2,2n}x_{2n} \\
\Phi_3: & A_{3,1}x_1&+ &A_{3,2}x_2 &+ & A^3x_3&+&A_{34}x_4&+& \cdots &+&A_{3,2n}x_{2n}\\
\Phi_4: & & & & & & &A_{44}x_4&+& \cdots &+&A_{4,2n}x_{2n}\\
\vdots & & & & & & & \vdots & \vdots & \vdots & \vdots\\
\Phi_{2n}: & & & & & & &A_{2n,4}x_4&+& \cdots &+&A_{2n,2n}x_{2n}\\
\end{array}
\end{equation}
where $A, A_{i,j}\in \mathbb R$.\par\noindent
If we assume that $\Phi^*([ \theta_2+ \widetilde{c}_3 \theta_3+\widetilde{c}_4 \theta_4]_{W_9})= [-\theta_2+ \widetilde{b}_3 \theta_3+\widetilde{b}_4 \theta_4]_{W_9}$, then \par\noindent
 $A^{8}dx_1\wedge dx_3|_0=-dx_1\wedge dx_3|_0$, which is a contradiction.
\end{proof}

\smallskip

\medskip

Statement $(ii)$ of Theorem \ref{klasw9} follows from the conditions
in the proof of part $(i)$ (codimension) and  from Theorem \ref{thm B} and Proposition \ref{sm} (symplectic multiplicity) and Proposition \ref{ii} (index of isotropy).

\medskip

 Statement $(iii)$ of Theorem \ref{klasw9} follows from Theorem \ref{geom-cond-w9}. The singularity classes corresponding to normal forms are disjoint because they can be distinguished by geometric conditions.

\smallskip

To prove statement $(iv)$ of Theorem \ref{klasw9} we have to show that the parameters $c, c_1, c_2$ are moduli in the normal forms. The proofs are very similar in all cases. We consider as an example the normal form with two parameters $[\theta_1+c_1\theta_2+c_2\theta_3]_{W_9}$. From Table \ref{infini-w9} we see that the tangent space to the orbit
of $[\theta_1+c_1\theta_2+c_2\theta_3]_{W_9}$ at $[\theta_1+c_1\theta_2+c_2\theta_3]_{W_9}$ is spanned by the linearly independent algebraic restrictions
$[7\theta_1+8c_1\theta_2+9c_2\theta_3]_{W_9}$, $[\theta_4]_{W_9},[\theta_5]_{W_9}, [\theta_6]_{W_9}, [\theta_7]_{W_9}, [\theta_8]_{W_9}, [\theta_9]_{W_9}.$ Hence the algebraic restrictions
$[\theta_2]_{W_9}$ and $[\theta_3]_{W_9}$ do not belong to it. Therefore the parameters $c_1$ and $c_2$ are independent moduli in the normal form
$[\theta_1+c_1\theta_2+c_2\theta_3]_{W_9}$.
\bigskip


\bibliographystyle{amsalpha}

\end{document}